\documentclass{article}
\usepackage{graphicx} 
\usepackage{amsthm}
\usepackage{amsfonts, amsmath, amssymb, amsthm}
\usepackage{cleveref}
\usepackage[utf8x]{inputenc}

\author{Quentin Le Hou\'erou \and Ludovic Patey}

\title{$\Pi^0_4$ conservation of a Carlson-Simpson lemma\\ for 1-variable words}

\newtheorem{theorem}{Theorem}
\numberwithin{theorem}{section}
\newtheorem{maintheorem}[theorem]{Main Theorem}
\newtheorem{lemma}[theorem]{Lemma}
\newtheorem{question}[theorem]{Question}
\newtheorem{proposition}[theorem]{Proposition}
\newtheorem{statement}[theorem]{Statement}

\newtheorem{definition}[theorem]{Definition}
\newtheorem{corollary}[theorem]{Corollary}

\makeatletter
\newtheorem*{rep@theorem}{\rep@title}
\newcommand{\newreptheorem}[2]{%
\newenvironment{rep#1}[1]{%
 \def\rep@title{#2 \ref{##1}}%
 \begin{rep@theorem}}%
 {\end{rep@theorem}}}
\makeatother

\newreptheorem{theorem}{Theorem}
\newreptheorem{maintheorem}{Main Theorem}

\usepackage{xcolor}

\newcommand{\PIOOCA}{\Pi^1_1\mathsf{-CA}}
\newcommand{\RCA}{\mathsf{RCA}}
\newcommand{\WKL}{\mathsf{WKL}}

\newcommand{\COH}{\mathsf{COH}}
\newcommand{\ACA}{\mathsf{ACA}}
\newcommand{\OVW}{\mathsf{OVW}}
\newcommand{\OVWD}[3]{\mathsf{OVW}^{#1}_{#3}(#2)}
\newcommand{\CSL}{\mathsf{CSL}}
\newcommand{\CSLD}[3]{\mathsf{CSL}^{#1}_{#3}(#2)}
\newcommand{\CTD}[3]{\mathsf{CT}^{#1}_{#3}(#2)}
\newcommand{\GR}{\mathsf{GR}}
\newcommand{\OGR}{\mathsf{OGR}}
\newcommand{\RFCSL}{\mathsf{LCSL}}

\newcommand{\TT}{\mathsf{TT}}
\newcommand{\FUT}{\mathsf{FUT}} 
\newcommand{\HT}{\mathsf{HT}} 

\newcommand{\PRA}{\mathsf{PRA}}

\newcommand{\BSig}{\mathsf{B}\Sigma^0}
\newcommand{\ISig}{\mathsf{I}\Sigma^0}

\newcommand{\RT}{\mathsf{RT}}

\newcommand{\Psf}{\mathsf{P}}
\newcommand{\Qsf}{\mathsf{Q}}

\newcommand{\HH}{\mathbb{H}}
\newcommand{\NN}{\mathbb{N}}
\newcommand{\QQ}{\mathbb{Q}}

\newcommand{\GG}{\mathbb{G}}

\newcommand{\Mk}{\mathfrak{M}}
\newcommand{\Nk}{\mathfrak{N}}
\newcommand{\Fk}{\mathfrak{F}}

\renewcommand{\P}{\mathcal{P}}
\newcommand{\M}{\mathcal{M}}
\renewcommand{\P}{\mathcal{P}}

\newcommand{\uh}[0]{{\upharpoonright}}

\newcommand{\OSub}{\mathsf{OSub}}
\newcommand{\USub}{\mathsf{USub}}

\newcommand{\set}{\mathsf{set}}
\newcommand{\FU}{\mathsf{FU}} 
\newcommand{\occ}{\mathsf{occ}} 

\newcommand{\BCSL}{\mathsf{BCSL}}
\newcommand{\BlockCSL}[1]{\BCSL^1\text{-}#1\text{-large}(\theta)}

\newcommand{\Cod}{\mathsf{Cod}}

\newcommand{\finsub}{\subseteq_{\mathtt{fin}}}

\newcommand{\bstr}{2^{<\NN}}

\makeatletter
\DeclareFontEncoding{LS2}{}{\@noaccents}
\makeatother
\DeclareFontSubstitution{LS2}{stix}{m}{n}
\DeclareSymbolFont{arrows3}{LS2}{stixtt}{m}{n}
\DeclareMathSymbol{\smashtimes}{\mathbin}{arrows3}{"A4}

\usepackage{mathabx}

\DeclareSymbolFont{bbold}{U}{bbold}{m}{n} 
\DeclareMathSymbol{\bbomega}{\mathord}{bbold}{"7F}

\def\qt#1{``#1''}%

\begin{document}

\maketitle

\begin{abstract}
Carlson and Simpson proved that for every finite coloring of the 1-variable words over a finite alphabet~$A$, there is an infinite $\omega$-variable word on which all the 1-variable words are monochromatic. This statement for $\ell$-colorings, written $\CSL^1_\ell$, is known to be strictly weaker than~$\ACA_0$. We prove that $\RCA_0 + \CSL^1_2$ is a $\forall \Pi^0_4$-conservative extension of~$\RCA_0 + \BSig_2$. Among its consequences, it implies that neither the indivisibility of the universal triangle-free Henson graph for 2-colorings, nor the tree theorem for pairs and two colors, imply $\Sigma^0_2$-induction. This answers a question of Chong, Li, Wang and Yang~\cite{chong2020strength}.
\end{abstract}

\section{Introduction}

A \emph{1-variable word} over a finite alphabet~$A$ is a finite word over the alphabet $A \sqcup \{\star\}$ where the variable~$\star$ appears at least once. An \emph{$\omega$-variable word} over~$A$ is an infinite sequence over the alphabet $A \sqcup \{ x_0, x_1, \dots \}$ where each variable kind~$x_i$ appears at least once, and the first occurrence of~$x_i$ is before the first occurrence of~$x_{i+1}$. Given an $\omega$-variable word~$W$ and a 1-variable word~$v$ of length~$|v|$, we denote by $W[v]$ the 1-variable word obtained by replacing every occurrence of $x_i$ in~$W$ by the $i$th element of~$v$, and cutting at the first occurrence of the variable~$x_{|v|}$.
The Carlson-Simpson Lemma~\cite{carlson1984dual} for 1-variable words ($\CSL^1_\ell$) is the following statement. See \Cref{sect:ramsey-variable-words} for a formal definition.
\begin{quote}
For every $\ell$-coloring of the 1-variable words over a finite alphabet~$A$, there is an $\omega$-variable word~$W$ over~$A$ and some color~$c < \ell$ such that, for every 1-variable word~$v$ over~$A$, $W[v]$ has color~$c$.
\end{quote}
The Carlson-Simpson Lemma is of interest for its connections with the Dual Ramsey theorem~\cite{carlson1984dual}, the existence of big Ramsey degrees for universal structures~\cite{hubivcka2020big}, and its relation to Hindman's theorem~\cite{Hindman1974Finite}. In particular, $\CSL^1_\ell$ restricted to unary alphabets can be formulated as follows, where $[\NN]^{<\omega}$ denotes the collection of all finite sets of integers.
\begin{quote}
For every coloring $f : [\NN]^{<\omega} \to \ell$, there is an infinite sequence of non-empty (finite or infinite) sets $X_0, X_1, \dots$ with $\min X_i < \min X_{i+1}$ and some color~$c < \ell$ such that for every finite set~$F \subseteq \NN$ of size at least~2, $f(\{0, \dots, \min X_{\max F}\} \cap \bigcup_{i \in F} X_i) = c$.
\end{quote}
Note that if we furthermore require that $\max X_i < \min X_{i+1}$, then each~$X_i$ is finite, and the resulting statement is equivalent to Hindman's theorem (see \Cref{prop:fut-ovw1}).

We are interested in the meta-mathematics of the Carlson-Simpson Lemma for 1-variable words from a reverse-mathematical viewpoint. Reverse mathematics is a foundational program whose goal is to find optimal axioms to prove theorems from ordinary mathematics. It uses the framework of subsystems of second-order arithmetic, with a base theory, $\RCA_0$, capturing \emph{computable mathematics}. The theory $\RCA_0$ consists of the axioms of Robinson arithmetic, together with the $\Delta^0_1$-comprehension scheme and $\Sigma^0_1$-induction scheme. See Simpson~\cite{simpson2009subsystems} or Dzhafarov and Mummert~\cite{dzhafarov2022reverse} for an introduction to reverse mathematics.

Among the systems of interest, $\WKL_0$ is the extension of $\RCA_0$ with K\"onig's lemma for binary trees, $\ACA_0$ is the extension of~$\RCA_0$ with the arithmetic comprehension scheme, or equivalently with the statement $\forall X \exists Y(Y = X')$, where $X' = \{ e : \Phi^X_e(e)\downarrow \}$ is the Turing jump of~$X$. Last, $\ACA_0^+$ is the extension of $\RCA_0$ with the statement $\forall X \exists Y(Y = X^{(\omega)})$, where $X^{(\omega)}$ is the $\omega$-jump of $X$ defined inductively as follows: $X^{(0)} = X$, $X^{(n+1)} = (X^{(n)})'$, and  $X^{(\omega)} = \bigoplus_n X^{(n)}$. These systems are linearly ordered by logical strength:
$$
\RCA_0 < \WKL_0 < \ACA_0 < \ACA_0^+
$$
The system $\ACA_0^+$ is famous for being the best known upper bound to Hindman's theorem~\cite{Towsner2012simple} and, until recently, to the Carlson-Simpson Lemma for 1-variable words~\cite{angles2023carlson}. On the other hand, Hindman's theorem implies~$\ACA_0$ over~$\RCA_0$~\cite{blass1987logical}.

In this article, we prove the following theorem, where a $\forall \Pi^0_n$-formula is of the form $\forall X \varphi(X)$ for some $\Pi^0_n$-formula~$\varphi$, and $\BSig_n$ denotes the collection (or bounding) scheme for $\Sigma^0_n$-formulas.

\begin{maintheorem}\label[maintheorem]{thm:csl12-pi04-conservation}
$\WKL_0 + \CSL^1_2$ is $\forall \Pi^0_4$-conservative over~$\RCA_0 + \BSig_2$.
\end{maintheorem}

As $\RCA_0 \vdash \CSL^1_2 \to \CSL^1_\ell$ for every standard $\ell \in \omega$, it follows that $\WKL_0 + \CSL^1_\ell$ is also $\forall \Pi^0_4$-conservative over~$\RCA_0 + \BSig_2$. On the other hand, the statement $\forall \ell \CSL^1_\ell$ implies $\BSig_3$ over~$\RCA_0$, and therefore is not even $\Pi_1$-conservative over~$\RCA_0 + \BSig_2$ (see \Cref{prop:csl1-bsig3}).
By \Cref{thm:csl12-pi04-conservation}, $\WKL_0 + \CSL^1_2$ does not imply~$\ACA_0$ or even $\Sigma^0_2$-induction. The former fact was proven by Liu and Patey~\cite{liu2026reverse}, and shows a fundamental difference with Hindman's theorem. In particular, since $\WKL_0 + \CSL^1_2$ proves the indivisibility of the universal triangle-free Henson graph for 2-colorings, it follows that the latter does not imply $\Sigma^0_2$-induction (see \Cref{sec:triangle-free}).

\subsection{Organization of the paper}

The proof of \Cref{thm:csl12-pi04-conservation} uses a refinement of the indicator method of Kirby and Paris~\cite{kirby1977initial} using a parameterized variant of Ketonen and Solovay's $\alpha$-largeness~\cite{ketonen1981rapidly}. In \Cref{sect:ramsey-variable-words}, we review the Ramsey theory of variable words, and their analysis in reverse mathematics. Then, in \Cref{sec:parameterized-largeness}, we introduce the framework of parameterized largeness. In \Cref{sec:largeness-gr1} and \Cref{sec:largeness-csl1}, we prove the combinatorial core of our theorem by showing that this notion of largeness is closed under applications of the Graham-Rothschild theorem and the Carlson-Simpson theorem for 1-variable words. In \Cref{sect:conservation}, we prove \Cref{thm:csl12-pi04-conservation} using these combinatorial results, and study its relations with the indivisibility of the universal triangle-free graph and the tree theorem for pairs. Last, in \Cref{sect:open-questions}, we conclude with open questions.

\section{Ramsey theory of variable words}\label[section]{sect:ramsey-variable-words}

In this section, we formally introduce the concepts and theorems from Ramsey theory that will be used throughout the paper\footnote{There is a difference of terminology between the literature from combinatorics and from reverse mathematics. In combinatorics, an $n$-variable word is ordered by default, while an unordered $n$-variable word is called an $n$-parameter word. The Ordered Variable Word theorem studied in reverse mathematics is called the Carlson-Simpson theorem in combinatorics, which should not be confused with the Carlson-Simpson Lemma studied in reverse mathematics. The latter corresponds to the unordered version of the Ordered Variable Word theorem.}. See Graham, Rothschild and Spencer~\cite{graham2013ramsey} or Todorcevic~\cite{todorcevic2010introduction} for general introduction to Ramsey theory, and Dodos and Kanellopoulos~\cite{dodos2016ramsey} for a particular focus to Ramsey's theory for product spaces.

\subsection{Variable words}\label[section]{sect_intro-vw}

We identify a non-negative integer $k \in \NN$ with the set $\{0, \dots, k-1\}$.
A \emph{word} over a finite alphabet $A$ is a finite ordered sequence $w = \langle a_0, \dots, a_{n-1} \rangle \in A^n$ for some $n \in \NN$. An \emph{infinite word} over $A$ is a function $W : \NN \to A$. We denote by $A^{<\omega}$ and $A^\omega$ the sets of finite and infinite words over $A$, respectively. For $w = \langle a_0, \dots, a_{n-1} \rangle$, we write $|w|$ for the \emph{length} $n$ of the word~$w$ and given $i < n$, we let $w(i) = a_i$.

An \emph{$n$-variable word} over $A$ is a word $w$ over the alphabet $A \sqcup \{ x_j : j < n \}$ where each $x_j$ appears at least once and the first occurrence of $x_j$ appears before the first occurrence of $x_{j+1}$. If the last occurrence of $x_j$ appears before the first occurrence of~$x_{j+1}$, we say that $w$ is an \emph{ordered $n$-variable word}. We call $n$ the \emph{dimension} of the $n$-variable word $w$. 
Denote by $A^{<\omega, n}$ and $A^{<\omega, n}_{<}$ the sets of unordered and ordered $n$-variable words over $A$, respectively. Note that $A^{<\omega} = A^{<\omega, 0}$, and that $A^{<\omega, n} \subseteq (A \sqcup \{x_0, \dots, x_{n-1}\})^{<\omega}$.

An \emph{$\omega$-variable word} over $A$ is an infinite word $W$ over the alphabet $A \sqcup \{ x_j : j \in \NN \}$ where each variable kind $x_j$ appears at least once and the first occurrence of $x_j$ appears before the first occurrence of $x_{j+1}$. The notion of \emph{ordered $\omega$-variable word} is defined accordingly. Note that each variable kind in an $\omega$-variable word can appear infinitely often, while they occur necessarily finitely often in an ordered $\omega$-variable word. We write $A^{\omega, \omega}$ and $A^{\omega, \omega}_{<}$ for the set of all unordered and ordered $\omega$-variable words over $A$, respectively.

Given an (ordered) $\omega$-variable word $W$ over $A$ and an (ordered) $n$-variable word $u$ over $A$, we write $W[u]$ for the finite (ordered) $n$-variable word over $A$ where each occurrence of $x_j$ is replaced by $u(j)$, and cut before the first occurrence of $x_{|u|}$. In particular, letting $\epsilon$ be the empty word, $W[\epsilon]$ is the initial segment of $W$ before the first occurrence of $x_0$. The substitution notation $W[a]$ must not be confused with $W(i)$: the former notation denotes the finite word obtained by substitution of all the occurrences of~$x_0$ by the letter~$a$ and cutting before the first occurrence of~$x_1$, while the latter notation denotes the letter in $W$ at position~$i$. The variable kind universe $V = \{ x_0 < x_1 < \dots \}$ might differ between variable words to avoid name clash or for simplicity. It might be with other symbols $\{ y_0 < y_1 < \dots \}$ or range over a subset $\{ x_{i_0} < x_{i_1} < \dots \}$. It will always be disjoint from the alphabet~$A$.


Given $n \in \NN$ and an $\omega$-variable word~$W$, we let
$$
\USub^{n,\star}_{A}(W) = \{ W[u] : u \in A^{<\omega, n} \}
$$
The ordered counterpart of $\USub^{n,\star}_A(W)$ is written $\OSub^{n,\star}_{A}(W)$.

\subsection{Infinite variable word theorems}

Carlson and Simpson~\cite{carlson1984dual} studied a dual version of Ramsey's theorem, and proved for this a purely combinatorial statement about colorings of variable words, known in reverse mathematics as the Carlson-Simpson Lemma:

\begin{theorem}[Higher-order Carlson-Simpson Lemma]\label[theorem]{thm_ho-csl}
Fix $n \geq 0$ and $\ell \geq 1$.
For every finite alphabet $A$ and every finite partition $C_0 \sqcup \dots \sqcup C_{\ell-1} = A^{<\omega, n}$, 	
there is some color $i < \ell$ and an infinite $\omega$-variable word $W$ such that $\USub^{n,\star}_A(W) \subseteq C_i$.
\end{theorem}

We write $\CSLD n k \ell$ for the statement of \cref{thm_ho-csl} restricted to $\ell$-colorings of $n$-variable words over finite alphabets of size $k$. If a parameter is omitted, then it will be assumed to be universally quantified. Thus, $\CSL^1_2$ is the statement $\forall k\CSL^1_2(k)$. The Higher-order Carlson-Simpson Lemma admits a simple inductive proof on~$n$, starting from its zero-dimensional version~\cite{carlson1984dual}. 

The Carlson-Simpson lemma belongs to a long line of tree partition theorems, such as the tree theorem~\cite{chubb2009reverse,corduan2010reverse,patey2016strength,chong2020strength,chong2023conservation} and Milliken's tree theorem~\cite{milliken1979ramsey,angles2020milliken}, among others. These partition theorems serve as a backbone of many other combinatorial theorems, such as statements of the existence of finite big Ramsey degrees. In particular, $\CSL$ was used to reprove a theorem by Dobrinen~\cite{dobrinen2020ramsey} stating that the universal triangle-free Henson graph admits finite big Ramsey degree~\cite{hubivcka2020big,angles2023carlson}.

The higher-order Carlson-Simpson lemma admits an ordered variable word counterpart, whose case $n = 0$ was proven by Carlson and Simpson~\cite{carlson1984dual}, and which follows from Carlson's theorem~\cite{carlson1988some} in the general case:

\begin{theorem}[Higher-order Ordered Variable Word theorem]\label[theorem]{thm_ho-ovw}
Fix $n \geq 0$ and $\ell \geq 1$.
For every finite alphabet $A$ and every finite partition $C_0 \sqcup \dots \sqcup C_{\ell-1} = A^{<\omega, n}_{<}$, 	
there is some color $i < \ell$ and an infinite ordered $\omega$-variable word $W$ such that $\OSub^{n,\star}_{A}(W) \subseteq C_i$.
\end{theorem}

We write $\OVWD n k \ell$ for the statement of \cref{thm_ho-ovw} restricted to $\ell$-colorings of ordered $n$-variable words over finite alphabets of size $k$. 

The statement $\OVWD 0 k \ell$ was originally proven using a Baumgartner-style argument~\cite{carlson1984dual}, and admits multiple quantitative refinements, such as a density version~\cite{dodos2014density} and a piecewise syndetic version. More recently, Liu and Patey~\cite{liu2026reverse} gave a proof of $\OVWD 0 k \ell$ following the style of Towsner~\cite{Towsner2012simple}.

The statement $\OVWD 1 1 \ell$ is a slight variation of the Finite Union theorem, or equivalently Hindman's theorem~\cite{Hindman1974Finite}, which is also known to admit a purely combinatorial version due to Towsner~\cite{Towsner2012simple}. 
On the other hand, the general statement is a consequence of Carlson's theorem for extracted variable words~\cite{carlson1988some} whose only known proofs use topological dynamics or ultrafilters, except for its zero-dimensional version~\cite{bompard2025reverse}. Note that Carlson's theorem admits an inductive proof using its one-dimensional version as base case, so the existence of a combinatorial proof of Carlson's theorem and of $\OVWD n k \ell$ is reduced to the existence of a combinatorial proof of $\CTD 1 k \ell$. Thankfully, the known applications of the higher-order variable word theorems are already consequences of their unordered versions.

\subsection{Finite variable word theorems}

All these variable word statements admit finitary counterparts thanks to a compactness argument. There exist also direct proofs in the spirit of the celebrated Hales-Jewett theorem. Following the new proof of the Hales-Jewett theorem by Shelah~\cite{shelah1988primitive}, all the considered finitary theorems admit primitive recursive bounds~\cite{dodos2016ramsey}. 

Given an (ordered) $n$-variable word~$w$ over~$A$ and an (ordered) $m$-variable word~$u$ over~$A$, with~$|u| \leq n$, we write $w[u]$ for the finite (ordered) $m$-variable word over~$A$ where each occurrence of~$x_j$ is replaced by $u(j)$, and cut before the first occurrence of $x_{|u|}$ if $|u| < n$. If $|u| = n$, then $w[u]$ is not cut and $|w[u]| = |w|$. 
Given $m \leq n \in \NN$ and a $n$-variable word~$w$, we let
$$
\USub^{m,\star}_{A}(w) = \{ w[u] : u \in A^{<n, m} \} \hspace{10pt}\mbox{ and }\hspace{10pt} \USub^{m,=}_{A}(w) = \{ w[u] : u \in A^{n, m} \}
$$
We also write $\USub^{=}_A(w)$ for $\bigcup_{m \leq n} \USub^{m,=}_A(w)$.
The ordered counterpart of $\USub^{\square}_A(w)$ is written $\OSub^{\square}_{A}(w)$, were $\square$ is a placeholder for the various notations.

We shall be interested in the two following finitary theorems:
The first theorem is the Graham-Rothschild theorem about unordered variable words~\cite{graham2013ramsey}, whose primitive recursive bounds were proven by Shelah~\cite{shelah1988primitive}. There exists an ordered variable word counterpart, proven by Dodos and Kanellopoulos~\cite[Theorem 2.15]{dodos2016ramsey}. 


\begin{theorem}[Graham-Rothschild, $\RCA_0$]\label[theorem]{thm:graham-rothschild}
There exists a primitive recursive function~$GR(k, d, m, \ell)$ such that for every $n \geq GR(k, d, m, \ell)$, every $n$-variable word $w$ over an alphabet~$A$ of size~$k$ and every coloring $f : \USub^{m,=}_A(w) \to \ell$, there exists some $v \in \USub^{d,=}_A(w)$ such that $\USub^{m,=}_A(v)$ is $f$-monochromatic.
\end{theorem}


The second theorem is the finitary version of the Ordered Variable Word theorem, proven by Dodos, Kanellopoulos and Tyros~\cite{dodos2014density}. 

\begin{theorem}[Finite OVW, $\RCA_0$]\label[theorem]{thm:finite-ovw}
There exists a primitive recursive function~$OVW(k, d, m, \ell)$ such that for every $n \geq OVW(k, d, m, \ell)$, every ordered $n$-variable word $w$ over an alphabet~$A$ of size~$k$ and every coloring $f : \OSub^{m,\star}_{A}(w) \to \ell$, there exists some $v \in \OSub^{d,=}_{A}(w)$ such that $\OSub^{m,\star}_{A}(v)$ is $f$-monochromatic.
\end{theorem}

We are actually interested in the unordered counterpart of \Cref{thm:finite-ovw} for $m = 0$ and $m = 1$.
The unordered version of \Cref{thm:finite-ovw} does not seem to appear in the literature in its full generality. Thankfully, the cases $m = 0$ and $m = 1$ are immediate consequences of \Cref{thm:finite-ovw} since $1$-variable words and ordered $1$-variable words coincide (see the paragraph after \Cref{lem:csl-large-canon}).

\subsection{Variable word theorems in reverse mathematics}

Theorems about variable words have been extensively studied in reverse mathematics, due to their connection with many other combinatorial theorems. We now review the reverse-mathematical literature around these statements and their relations to variable word theorems.

\textit{Dual Ramsey theorem and variable words.} In order to prove a dual version of Ramsey's theorem, Carlson and Simpson~\cite{carlson1984dual} proved the zero-dimensional version of the Ordered Variable Word theorem~\cite[Theorem 6.3]{carlson1984dual} and used it as a base principle to inductively prove the now-called Carlson-Simpson Lemma~\cite[Lemma 2.4]{carlson1984dual}. Miller and Solomon~\cite{miller2004effectiveness} studied both the Dual Ramsey theorem and $\OVW^0$ from the viewpoint of reverse mathematics. They showed that $\OVW^0_2(2)$ is not provable over~$\WKL_0$. Dzhafarov, Flood, Solomon and Westrick~\cite{Dzhafarov2017Effectiveness} studied thoroughly the Dual Ramsey theorem and proved that $\CSL$ holds over~$\PIOOCA_0$. Liu, Monin and Patey~\cite{liu2019computable} proved that $\OVW^0(2)$ holds over~$\ACA$, a strengthening of $\ACA_0$ with full induction. They also showed that $\OVW^0(2)$ follows from Hindman's theorem over~$\RCA_0$. Later, Anglès d'Auriac, Liu, Mignoty and Patey~\cite{angles2023carlson} proved that $\OVW^0$ holds over~$\ACA_0$ and deduced that for every~$n \in \omega$, the statement $\CSL^n$ holds over~$\ACA_0^+$, the strengthening of $\ACA_0$ with the axiom stating the existence of the $\omega$-jump of any set. Le Houérou, Patey and Yokoyama~\cite{houerou2024pi} proved that $\RCA_0 + \OVW^0$ is $\forall \Pi^0_4$-conservative over~$\RCA_0 + \BSig_2$, from which it follows that $\OVW^0$ does not imply $\ACA_0$ over~$\RCA_0$. More recently, Liu and Patey~\cite{liu2026reverse} gave a Towsner-style proof of $\OVW^0$ and showed that it implies neither $\ACA_0$, nor Ramsey's theorem for pairs over~$\RCA_0$. They also proved that for every~$n \in \omega$, the statement $\CSL^n$ holds over~$\ACA_0$.

\textit{Hindman's theorem and variable words.} Hindman's theorem ($\HT$) is notoriously equivalent to the Finite Union theorem ($\FUT$), which states, for every~$\ell \geq 1$ and for every coloring $f : [\NN]^{<\omega} \to \ell$, the existence of an infinite sequence of finite sets $\vec{F} = F_0 < F_1 < \dots$ such that $\FU(\vec{F})$ is monochromatic. Here, $\FU(\vec{F})$ denotes the set of all non-empty finite unions over~$\vec{F}$. The Finite Union theorem is closely related to the 1-dimensional version of the Ordered Variable Word theorem for unary alphabets. Indeed, $\OVW^1_\ell(1)$ can be reformulated as follows: \qt{For every coloring $f : [\NN]^{<\omega} \times \NN \to \ell$, there is an infinite sequence of finite sets $\vec{F} = F_0 < F_1 < \dots$ and a color~$i < \ell$ such that for every~$E \in \FU(\vec{F})$, and every~$s$ such that $\max E < \min F_s$, $f(E, \min F_s) = i$.} It is not difficult to see that $\FUT_\ell$ and $\OVW^1_\ell(1)$ are equivalent over~$\RCA_0$ (see \Cref{prop:fut-ovw1}). Hindman's theorem admits multiple proofs, which were systematically analyzed in reverse mathematics: First, the original proof by Hindman~\cite{Hindman1974Finite}, which was later simplified by Baumgartner~\cite{baumgartnerShortProofHindmans1974}. Both proofs were analyzed by Blass, Hirst, and Simpson~\cite{blass1987logical}, who proved that the former holds over~$\ACA_0^+$, while the latter holds over the much stronger system~$\Pi^1_2\mbox{-}\mathsf{TI}_0$. In terms of lower bounds, they also proved that $\HT_2$ implies $\ACA_0$ over~$\RCA_0$. It follows from the equivalence between $\HT_2$ and $\OVW^1_2(1)$ that $\OVW^1_2(1)$ implies $\ACA_0$ over~$\RCA_0$. Galvin and Glazer~\cite{hindmanAlgebraStoneCechCompactification2012} proved Hindman's theorem using idempotent ultrafilters. Their proof, analyzed by Towsner~\cite{towsner2011hindman} holds in $\Sigma^1_1\mbox{-}\mathsf{TI}_0$. Last, Towsner~\cite{Towsner2012simple} gave a simple combinatorial proof of Hindman's theorem, which holds over~$\ACA_0^+$.

\textit{Big Ramsey degrees and tree partition theorems.} Structural Ramsey theory studies extensions of Ramsey's theorem to infinite mathematical structures. An infinite structure $\Mk$ admits \emph{finite big Ramsey degree} if for every finite sub-structure $\Fk$, there is some number~$k_{\Fk}$ such that for every finite coloring of $[\Mk]^\Fk$ (the isomorphic copies of~$\Fk$ in~$\Mk$), there is a sub-copy~$\Nk$ of~$\Mk$ such that $[\Nk]^{\Mk}$ uses at most~$k_{\Fk}$ colors. For instance, Devlin's theorem~\cite{devlin1980some} states that $(\QQ, <)$ admits finite big Ramsey degree, while Sauer~\cite{sauer2006coloring} proved that the random graph admits finite big Ramsey degree. The existence of finite big Ramsey degrees is usually proven using tree partition theorems. For instance, the two previous theorems are reduced to Milliken's tree theorem~\cite{angles2020milliken}, which holds in~$\ACA_0$. The existence of big Ramsey degrees for the universal triangle-free Henson graph was proven by Dobrinen~\cite{dobrinen2020ramsey} using forcing with strong coded trees. Hubi\v{c}ka~\cite{hubivcka2020big} gave another proof using the Carlson-Simpson Lemma. This was used by Anglès d'Auriac, Liu, Mignoty and Patey~\cite{angles2023carlson} to show that the existence of finite big Ramsey degrees for the universal triangle-free Henson graph holds in~$\ACA_0^+$. Recently, Cholak, Dobrinen and Towsner~\cite{cholak2026hensonupper} proved that it holds over~$\ACA_0$ and Cholak, Dobrinen and McCoy~\cite{cholak2026henson} showed that it implies~$\ACA_0$ as soon as we consider colorings of finite sub-structures of size at least~2. Liu and Patey~\cite{liu2026reverse} proved that the indivisibility of the universal triangle-free Henson graph admits cone avoidance, and therefore does not imply $\ACA_0$. In this paper, we show that this indivisibility statement for 2-colorings is $\forall \Pi^0_4$-conservative over~$\RCA_0 + \BSig_2$.

\subsection{Preliminary reductions}

We conclude this section with some simple reductions, both over $\RCA_0$ and over strong Weihrauch reducibility. The goal is to show to which extent \Cref{thm:csl12-pi04-conservation} is optimal: 
\begin{itemize}
    \item It does not hold when allowing more colors, even when restricted to unary alphabets: We shall see that $\forall \ell \CSL^1_\ell(1)$ implies $\BSig_3$ over~$\RCA_0$, and therefore is not even $\Pi_1$-conservative over $\RCA_0 + \BSig_2$.
    \item It does not hold when asking for the $\omega$-variable word to be ordered, even for unary alphabets. Indeed, $\OVW^1_2(1)$ implies $\ACA_0$ over~$\ACA_0$.
\end{itemize}

\begin{definition}
A problem~$\Psf$ is \emph{strongly Weihrauch reducible} to another problem $\Qsf$ (written $\Psf \leq_{sW} \Qsf$) if there are two Turing functionals $\Phi, \Psi$ such that for every~$\Psf$-instance~$X$, $\Phi^X$ is a $\Qsf$-instance~$\hat X$ such that for every~$\Qsf$-solution~$\hat Y$ to~$\hat X$, $\Phi^{\hat Y}$ is a $\Psf$-solution to~$X$.
\end{definition}

Given a set~$X$ and some~$n \in \NN$, we write $[X]^n$ for the set of all $n$-subsets of~$X$.
Ramsey's theorem for $n$-tuples and $\ell$-colorings ($\RT^n_\ell$) is the statement \qt{For every coloring $f : [\NN]^n \to \ell$, there is an infinite set $H \subseteq \NN$ such that $[H]^n$ is $f$-monochromatic.}
The following proposition was proven by Liu and Patey~\cite[Lemma 8.5]{liu2026reverse} using multiple reductions involving intermediary principles. We give a more direct proof:

\begin{proposition}\label[proposition]{prop:rt-to-csl}
For every $n, \ell \geq 0$, $\RT^{n+1}_\ell \leq_{sW} \CSL^n_\ell(1)$.
Moreover, $\RCA_0 \vdash \forall n, \ell (\CSL^n_\ell(1) \rightarrow \RT^{n+1}_\ell)$.
The same holds when replacing $\CSL^n_\ell(1)$ with $\OVW^n_\ell(1)$.
\end{proposition}
\begin{proof}
Fix~$n, \ell \geq 0$.
Let $f : [\NN]^{n+1} \to \ell$ be an instance of $\RT^{n+1}_\ell$. Let $A = \{a\}$ and define $g : A^{<\omega, n} \to \ell$ by letting $g(w) = f(i_0, \dots, i_{n-1}, |w|)$, 
where $i_s$ is position of the first occurrence of the variable~$x_s$ in~$w$. Let $W$ be an $\omega$-variable word over~$A$ such that $\USub^{n,\star}_A(W)$ is $g$-monochromatic for some color~$c < \ell$. For clarity, say that the variable kinds of~$W$ range over $\{y_0, y_1, \dots \}$. Let $H = \{h_0 < h_1 < \dots \}$ be such that $h_s$ is the position of  the first occurrence of~$y_s$ in~$W$. 

We claim that $[H]^{n+1}$ is $f$-monochromatic for color~$c$.
Let $\{h_{i_0} < \dots < h_{i_n} \} \in [H]^{n+1}$. Let $v$ be the variable word of length~$i_n$ defined by $v(s) = x_j$ if $s = i_j$, and $v(s) = a$ otherwise. Let $w = W[v]$. Note that $|w| = h_{i_n}$, and that the first occurrence of the variable~$x_j$ in~$w$ is at position $h_{i_j}$. By definition of~$g$, $g(w) = f(h_{i_0}, \dots, h_{i_n}) = c$.

The proof holds over~$\RCA_0$, and is exactly the same for ordered variable words. This completes our proof of \Cref{prop:rt-to-csl}.
\end{proof}

\begin{corollary}\label[corollary]{prop:csl1-bsig3}
$\RCA_0 \vdash \forall \ell \CSL^1_\ell(1) \rightarrow \BSig_3$. In particular, $\forall \ell \CSL^1_\ell(1)$ is not $\Pi_1$-conservative over~$\RCA_0 + \BSig_2$.
\end{corollary}
\begin{proof}
Hirst~\cite[Theorem 6.11]{hirst1987combinatorics} proved that $\RCA_0 \vdash \forall \ell \RT^2_\ell \to \BSig_3$.
It follows by \Cref{prop:rt-to-csl} that $\RCA_0 \vdash \forall \ell \CSL^1_\ell(1) \rightarrow \BSig_3$. Finally, $\RCA_0 + \BSig_3$ is not $\Pi_1$-conservative over~$\RCA_0 + \BSig_2$, as $\RCA_0 \vdash \BSig_3 \to \mathsf{Con}(\ISig_1)$ but $\RCA_0 + \BSig_2 \not \vdash \mathsf{Con}(\ISig_1)$, and $\mathsf{Con}(\ISig_1)$ is a $\Pi_1$-sentence.
\end{proof}

The next proposition clarifies the relationship between the ordered variable word theorem for unary alphabets, and the finite union theorem.

\begin{proposition}\label[proposition]{prop:fut-ovw1}
$\RCA_0 \vdash \forall \ell (\OVW^1_\ell(1) \leftrightarrow \FUT_\ell)$.
\end{proposition}
\begin{proof}
Fix~$\ell \geq 1$. Given a (finite or infinite) variable word~$w$ and some variable kind~$x_i$, we let $\occ(w, x_i)$ be the set of positions of all the occurrences of~$x_i$ in~$w$. Whenever~$w$ is a 1-variable word, we simply write $\occ(w)$.
\begin{itemize}
    \item $\OVW^1_\ell(1) \rightarrow \FUT_\ell$. Let $f : [\NN]^{<\NN} \to \ell$ be a coloring. Define the coloring $g : \{0\}^{<\omega,1} \to \ell$ as follows: $g(w) = f(\occ(w))$. Let~$W$ be an ordered $\omega$-variable word over~$\{0\}$ such that $\OSub^{1,\star}_{\{0\}}(W)$ is $g$-monochromatic for some color~$c < \ell$. For every~$s \in \NN$, let $F_s = \occ(W, x_s)$. Note that for every~$s \in \NN$, $\max F_s < \min F_{s+1}$ as $W$ is ordered. We claim that for every $H \in \FU(F_0, F_1, \dots )$, $f(H) = c$. Indeed, for every such~$H$, there is a 1-variable word~$w \in \OSub^{1,\star}_{\{0\}}(W)$ such that~$H = \occ(w)$. Then $f(H) = g(w) = c$.
    \item $\FUT_\ell \rightarrow \OVW^1_\ell(1)$. By Blass, Hirst, and Simpson~\cite{blass1987logical}, $\RCA_0 \vdash \FUT_\ell \to \ACA_0$. 
    Let $A = \{a\}$ be a unary alphabet, and $f : A^{<\omega, 1} \to \ell$ be a coloring. For every~$v \in A^{<\omega, 1}$ and $c < \ell$, let $R_{v,c}$ be the set of all~$n \in \NN$ such that $f(v \cdot a^{n-|v|}) = c$. Here, $a^x$ denotes the word of length~$x$ composed only of occurrences of the letter~$a$. 
    By $\COH$, which holds over~$\ACA_0$, there is an infinite $\vec{R}$-cohesive set~$C = \{ p_0 < p_1 < \dots \} \subseteq \NN$, that is, an infinite set $C \subseteq \NN$ such that for every $v \in A^{<\omega, 1}$ and $c < \ell$, either $C \subseteq^* R_{v,c}$, or $C \subseteq^* \overline{R}_{v,c}$. Here, $\subseteq^*$ means \qt{included up to finite changes}. Note that for every~$u_0, u_1 \in A^{<\omega,1}$ such that $\occ(u_0) = \occ(u_1)$, $u_0$ and $u_1$ are comparable, and for every~$c < \ell$, $R_{u_0,c} =^* R_{u_1,c}$, where $=^*$ means \qt{equal up to finite changes}. 
    Moreover, for every~$u \in A^{<\omega, 1}$, there is a unique $c < \ell$ such that $C \subseteq^* R_{u,c}$. It follows that for every finite~$H \subseteq C$, there is a unique color~$c < \ell$ such that for every 1-variable word~$u$ satisfying $\occ(u) = H$, $C \subseteq^* R_{u,c}$. Let $g : \P_{fin}(C) \to \ell$ be defined by $g(H) = c$ for this unique color~$c$. By $\FUT_\ell$, there is an infinite  sequence $F_0 < F_1 < \dots$ of finite subsets of~$C$ and some color~$c < \ell$ such that for every~$H \in \FU(F_0, F_1, \dots)$, $g(H) = c$. Let $W$ be the $\omega$-variable word defined for every~$s \in \NN$ by $W(s) = x_i$ if $s \in F_i$, and $W(s) = a$ otherwise. Then, for every~$u \in \OSub^{1,\star}_A(W)$, $\occ(u) \in \FU(F_0, F_1, \dots)$, hence $g(\occ(u)) = c$. By definition of~$g$, $C \subseteq^* R_{u, c}$, so $\lim_{n \in C} f(u \cdot a^{n-|u|}) = c$. By a greedy algorithm, one can $W \oplus f$-compute an $\omega$-variable word~$V \in \OSub_A(W)$ such that for every~$u \in \OSub^{1,\star}_A(V)$, $f(u) = c$.
\end{itemize}
\end{proof}

\begin{corollary}\label[corollary]{cor:ovw121-aca}
$\RCA_0 \vdash \OVW^1_2(1) \to \ACA_0$. In particular, $\OVW^1_2(1)$ is not $\Pi_1$-conservative over~$\RCA_0 + \BSig_2$.
\end{corollary}
\begin{proof}
By Blass, Hirst, and Simpson~\cite{blass1987logical}, $\RCA_0 \vdash \FUT_2 \to \ACA_0$, so we conclude by \Cref{prop:fut-ovw1}. Moreover, $\ACA_0 \vdash \mathsf{Con}(\ISig_1)$.
\end{proof}

Note the difference between $\OVW^1_2$ and $\CSL^1_2$, as the former implies $\ACA_0$ over~$\RCA_0$ by \Cref{cor:ovw121-aca}, while the latter is $\forall \Pi^0_4$-conservative over~$\RCA_0 + \BSig_2$ by \Cref{thm:csl12-pi04-conservation}.

\section{Largeness below $\bbomega^\bbomega$}\label[section]{sec:parameterized-largeness}

Ketonen and Solovay~\cite{ketonen1981rapidly} defined a notion of $\alpha$-largeness for $\alpha < \varepsilon_0$ quantifying the size of finite sets, to generalize the celebrated Paris-Harrington theorem. This notion of largeness was later used as a refinement of the Kirby and Paris~\cite{kirby1977initial} indicator method to produce $\forall\Pi^0_3$-conservation theorems over various subsystems of arithmetic~\cite{patey2018proof,kolo2020some,kolodziejczyk2023ramsey,towsner2024erdos}. More recently, the authors and Yokoyama defined a parameterized version of largeness to prove partial conservation theorems over $\RCA_0 + \BSig_2$ for $\forall \Pi^0_4$-sentences~\cite{houerou2023conservation,houerou2026conservationramsey}.

We now present this notion of largeness, parameterized by a fixed $\Delta^0_0$-formula $\theta(x, y, z)$.

\begin{definition}
Two finite sets~$X < Y$ are \emph{$\theta$-apart} if 
$$\forall x < \max X \exists y < \min Y \forall z < \max Y \theta(x, y, z)$$
\end{definition}

Note that $\theta$-apartness is a transitive relation. Moreover, if $X < Y$ are $\theta$-apart and $X_0 \subseteq X$ and $Y_0 \subseteq Y$, then $X_0, Y_0$ are $\theta$-apart. Contrary to Ketonen and Solovay's notion of largeness, we only consider ordinals~$\alpha$ of the form $\bbomega^n \cdot k$, with $n, k \in \NN$.

\begin{definition}[\cite{houerou2023conservation}]\label[definition]{defi:largeness-rca0-bsig2-t}
A set~$X \finsub \NN$ is 
\begin{itemize}
    \item \emph{$\bbomega^0$-large$(\theta)$} if $X \neq \emptyset$.
    \item \emph{$\bbomega^{(n+1)}$-large$(\theta)$} if $X \setminus \min X$ is  $(\bbomega^n \cdot \min X)$-large$(\theta)$
    \item \emph{$\bbomega^n \cdot k$-large$(\theta)$} if
there are $k$ pairwise $\theta$-apart $\bbomega^n$-large$(\theta)$ subsets of~$X$
$$
X_0 < X_1 < \dots < X_{k-1}
$$
\end{itemize}
\end{definition}

\begin{proposition}[\cite{houerou2023conservation}]\label[proposition]{prop:largeness-bsig2-largeness}
For every~$n \in \omega$,
$\RCA_0 + \BSig_2 + \forall x \exists y \forall z \theta(x, y, z)$ proves that for every~$b \geq 1$, $\omega^n \cdot b$-largeness$(\theta)$ is a largeness notion.
\end{proposition}

A set~$X$ is \emph{$\RT^n_k$-$\bbomega^n$-large$(\theta)$} if for every coloring $f : [X]^n \to k$, there is an $\bbomega^n$-large$(\theta)$ subset~$H \subseteq X$ such that $[H]^n$ is $f$-monochromatic. Le Houérou, Patey and Yokoyama~\cite{houerou2023conservation,houerou2026corrigendum} proved the following closure property:

\begin{theorem}[\cite{houerou2023conservation,houerou2026corrigendum}, $\RCA_0$]\label[theorem]{thm:rt22-bound}
There exists a primitive recursive function~$f_{\RT^2_2} : \NN \to \NN$ such that for every~$n \in \NN$, every $\bbomega^{f_{\RT^2_2}(n)}$-large$(\theta)$ set is $\RT^2_2$-$\bbomega^n$-large$(\theta)$.
\end{theorem}

In a follow-up article~\cite{houerou2026largeness}, the authors proved that such a function $f_{\RT^2_2}$ can be chosen to be polynomial.

\subsection{Basic combinatorics}

\begin{lemma}[\cite{houerou2024pi}, $\RCA_0$]\label[lemma]{lem:split-largeness}
For every $a,b$ and for every~$\omega^{a+2b+1}$-large$(\theta)$ set~$X$, then there are some $k \in \NN$ and some $\omega^a$-large$(\theta)$ pairwise $\theta$-apart subsets $X_0 < \dots < X_{k-1}$ of~$X$ such that every $H \in \prod_{i < k} X_i$ is $\omega^b$-large$(\theta)$ (here, we abuse the product notation and see $H$ as a subset of $\NN$ intersecting every $X_i$ exactly once rather than a tuple in the product $X_0 \times \dots \times X_{k-1}$).
\end{lemma}

Given an increasing function $g : \NN \to \NN$, a set $X \subseteq \NN$ is \emph{$g$-sparse} if for every~$x, y \in X$ with $x < y$, then $g(x) < y$. We shall only consider $g$-sparsity for primitive recursive functions~$g$. 

\begin{lemma}[Folklore, see Ketonen and Solovay~\cite{ketonen1981rapidly}]\label[lemma]{lem:fast-growing-size}
For every primitive recursive function $g : \NN \to \NN$, there exists some $a \in \omega$ such that every $\omega^a\mbox{-large}$ (and a fortiori $\omega^a\mbox{-large}(\theta)$) set $X$ satisfies $|X| > g(\min X - 1)$.
\end{lemma}

Combining \Cref{lem:fast-growing-size} with \Cref{lem:split-largeness}, one can always assume that an $\omega^n$-large$(\theta)$ set is $g$-large for some fixed primitive recursive function~$g$, at the cost of a polynomial overhead in the exponent:

\begin{lemma}[$\RCA_0$]\label[lemma]{lem:largeness-to-sparsity}
For every primitive recursive function~$g : \NN \to \NN$, there is some~$a \in \NN$ such that for every~$n \in \NN$,
every $\bbomega^{2n+a}$-large$(\theta)$ set contains a $g$-sparse $\bbomega^n$-large$(\theta)$ subset.
\end{lemma}
\begin{proof}
Fix~$g$. By \Cref{lem:fast-growing-size}, there is some~$a \in \NN$ such that every $\bbomega^a$-large$(\theta)$ set~$X$ satisfies $|X| > g(\min X-1)$. Let $X$ be an $\bbomega^{2n+a}$-large$(\theta)$ set. By \Cref{lem:split-largeness}, there are some $\bbomega^a$-large$(\theta)$ pairwise $\theta$-apart subsets $X_0 < \dots < X_{k-1}$ of~$X$ such that every~$H \in \Pi_{i < k} X_i$ is $\bbomega^n$-large$(\theta)$. In particular, the set $H = \{\max X_i : i < k\}$ is $\bbomega^n$-large$(\theta)$. Moreover, for every~$x, y \in H$ such that $x < y$, the interval $]x, y]$ contains an $\bbomega^a$-large$(\theta)$ set, so $g(x) < y$. It follows that $H$ is $g$-sparse.
\end{proof}

\subsection{Largeness and variable words}

Largeness is defined in terms of sets of integers, while the considered theorems are statements about words and variable words. We bridge the two notions by defining the notions of ordered and unordered $X$-variable word over a finite set~$X \subseteq \NN$.

\begin{definition}
Let~$w$ be an (ordered) $n$-variable word over a finite alphabet~$A$.
We write $\set(w)$ for the set of positions of first occurrences of the variable kinds.
We say that $w$ is \emph{$\alpha$-large$(\theta)$} if $\set(w)$ is $\alpha$-large$(\theta)$.
\end{definition}

Given a finite set~$X$, an (ordered) $|X|$-variable word~$w$ over~$A$ such that $\set(w) = X$ is called an \emph{(ordered) $X$-variable word}. We are now ready to define largeness versions of the Graham-Rothschild theorem, of the Ordered Variable Word theorem and the Carlson-Simpson lemma.

\begin{definition}
A set $X \subseteq_{fin} \NN$ is said to be \emph{$\GR^n_\ell(k)$-$\alpha$-large$(\theta)$} if for every $X$-variable word~$w$ over an alphabet $A$ of size $k$ and every coloring $f : \USub^{n,=}_A(w) \to \ell$, there exists some $\alpha$-large$(\theta)$ $v \in \USub^{=}_A(w)$ such that $\USub^{n,=}_A(v)$ is $f$-monochromatic. The set~$X$ is \emph{$\GR^n$-$\alpha$-large$(\theta)$ if it is $\GR^n_\ell(k)$-$\alpha$-large$(\theta)$ for every~$k, \ell < \min X$.}
\end{definition}

Given two sets $X_0 < X_1$ and an $X_0 \cup X_1$-variable word~$w$, it will often be useful to decompose $w$ into $w_0 \cdot w_1$ with $|w_0| < \min X_1$, where $w_0$ is an $X_0$-variable word and $w_1$ is a \emph{located} $X_1$-variable word. By \emph{located word}, we mean a word~$u$ whose domain is an interval in~$\NN$ rather than an initial segment of~$\NN$. In this case, we write $|u|$ for the least upper bound of its domain, and consider the concatenation of two located words only if the disjoint union of their domain forms an interval.

\begin{definition}
A set $X \subseteq_{fin} \NN$ is said to be \emph{$\OVW^n_\ell(k)$-$\alpha$-large$(\theta)$} if for every ordered $X$-variable word~$w$ over an alphabet $A$ of size $k$ and every coloring $f : \OSub^{n,\star}_{A}(w) \to \ell$, there exists some $\alpha$-large$(\theta)$ $v \in \OSub^{=}_{A}(w)$ such that $\OSub^{n,\star}_{A}(v)$ is $f$-monochromatic. The set~$X$ is \emph{$\OVW^n$-$\alpha$-large$(\theta)$ if it is $\OVW^n_\ell(k)$-$\alpha$-large$(\theta)$ for every~$k, \ell < \min X$.}
\end{definition}

\begin{definition}
A set $X \subseteq_{fin} \NN$ is said to be \emph{$\CSL^n_\ell(k)$-$\alpha$-large$(\theta)$} if for every $X$-variable word~$w$ over an alphabet $A$ of size $k$ and every coloring $f : \USub^{n,\star}_A(w) \to \ell$, there exists some $\alpha$-large$(\theta)$ $v \in \USub^{=}_A(w)$ such that $\USub^{n,\star}_A(v)$ is $f$-monochromatic. The set~$X$ is \emph{$\CSL^n$-$\alpha$-large$(\theta)$ if it is $\CSL^n_\ell(k)$-$\alpha$-large$(\theta)$ for every~$k, \ell < \min X$.}
\end{definition}

For every $n$-variable word~$w$ over an alphabet~$A$ of size~$k$, there exists a canonical isomorphism 
$$
\iota_w : \begin{cases} 
    [k]^{<n, d} \to \USub^{d,\star}_A(w) & \text{for every } d < n \\
    [k]^{n, d} \to \USub^{d,=}_A(w) & 
\end{cases}
$$
which maps the unique element of $[k]^{|w|,|w|}$ to $w$. Accordingly, if $w$ is an ordered $n$-variable word, there is an isomorphism 
$$
\iota_{w,<} : \begin{cases} 
    [k]^{<n, d}_{<} \to \OSub^{d,\star}_{A}(w) & \text{for every } d < n \\
    [k]^{n, d}_{<} \to \OSub^{d,=}_{A}(w) & 
\end{cases}
$$
Thanks to these isomorphisms, one can simplify the definition of the largeness versions of our theorems:

\begin{lemma}\label[lemma]{lem:gr-large-canon}
A set~$X = \{ x_0 < \dots < x_{a-1} \}$ is $\GR^n_\ell(k)$-$\alpha$-large$(\theta)$ iff for every coloring $f : [k]^{a,n} \to \ell$, there is some~$v \in \USub^{=}_{[k]}([k]^{a,a})$ such that $\{ x_i : i \in \set(v) \}$ is $\alpha$-large$(\theta)$ and $\USub^{n,=}_{[k]}(v)$ is $f$-monochromatic.
\end{lemma}

\begin{lemma}\label[lemma]{lem:ovw-large-canon}
A set~$X = \{ x_0 < \dots < x_{a-1} \}$ is $\OVW^n_\ell(k)$-$\alpha$-large$(\theta)$ iff for every coloring $f : [k]^{< a,n}_{<} \to \ell$, there is some~$v \in \OSub^=_{[k]}([k]^{a,a})$ such that $\{ x_i : i \in \set(v) \}$ is $\alpha$-large$(\theta)$ and $\OSub^{n,\star}_{[k]}(v)$ is $f$-monochromatic.
\end{lemma}

\begin{lemma}\label[lemma]{lem:csl-large-canon}
A set~$X = \{ x_0 < \dots < x_{a-1} \}$ is $\CSL^n_\ell(k)$-$\alpha$-large$(\theta)$ iff for every coloring $f : [k]^{< a,n} \to \ell$, there is some~$v \in \USub^=_{[k]}([k]^{a,a})$ such that $\{ x_i : i \in \set(v) \}$ is $\alpha$-large$(\theta)$ and $\USub^{n,\star}_{[k]}(v)$ is $f$-monochromatic.
\end{lemma}

Since the notions of ordered and unordered $n$-variable words coincide for $n = 0$ and $n = 1$, it follows from the previous lemmas that for $n \in \{0,1\}$, any $\OVW^n$-$\alpha$-large$(\theta)$ set is $\CSL^n$-$\alpha$-large$(\theta)$.
The authors~\cite{houerou2024pi} proved explicit bounds for $\GR^0$-$\alpha$-large$(\theta)$\footnote{Actually, the authors used an ordered version of the Graham-Rotschild theorem, say $\OGR^0$-$\alpha$-largeness$(\theta)$. Simply note that any $\OGR^0$-$\alpha$-large$(\theta)$ set is $\GR^0$-$\alpha$-large$(\theta)$.} and $\OVW^0$-$\alpha$-large$(\theta)$ for $\alpha$ for the form $\bbomega^n \cdot r$:

\begin{theorem}[\cite{houerou2024pi}, $\RCA_0$]\label[theorem]{thm:gr0-bound}
There exists a primitive recursive function $f_{\GR^0} : \NN \to \NN$ such that for every~$n \in \NN$, 
every $\bbomega^{f_{\GR^0}(n)}$-large$(\theta)$ set is $\GR^0$-$\bbomega^n$-large$(\theta)$.
\end{theorem}

\begin{theorem}[\cite{houerou2024pi}, $\RCA_0$]
There exists a primitive recursive function $f_{\OVW^0} : \NN \to \NN$ such that for every~$n \in \NN$, 
every $\bbomega^{f_{\OVW^0}(n)}$-large$(\theta)$ set is $\OVW^0$-$\bbomega^n$-large$(\theta)$.
\end{theorem}

\begin{corollary}\label[theorem]{thm:csl0-bound}
There exists a primitive recursive function $f_{\CSL^0} : \NN \to \NN$ such that for every~$n \in \NN$, 
every $\bbomega^{f_{\CSL^0}(n)}$-large$(\theta)$ set is $\CSL^0$-$\bbomega^n$-large$(\theta)$.
\end{corollary}
\begin{proof}
Let $f_{\CSL^0} = f_{\OVW^0}$ since any $\OVW^0$-$\bbomega^n$-large$(\theta)$ set is $\CSL^0$-$\bbomega^n$-large$(\theta)$.
\end{proof}

Let $X_0 < X_1$ be two sets, and $X = X_0 \cup X_1$. Let $w$ be an $X$-variable word over an alphabet~$A$. For clarity, we assume that the variables space ranges over $\{ x_n : n \in X \}$. We can write $w = w_0 \cdot w_1$, where $|w_0| < \min X$, so that $w_0$ is an $X_0$-variable word over~$A$ and $w_1$ is a located $X_1$-variable word over the extended alphabet~$B = A \sqcup \{ x_p : p \in X_0 \}$.
For every~$v_0 \in \USub^{0,=}_A(w_0)$, $v_0 \cdot w_1[x_p \mapsto v_0(p)]$ is an $X_1$-variable word over~$A$, where $w_1[x_p \mapsto v_0(p)]$ means that every occurrence of the variable~$x_p$ with~$p \in X_0$ is replaced with the value of~$v_0$ at position~$p$.

\section{Graham-Rothschild theorem}\label[section]{sec:largeness-gr1}

The goal of this section is to prove the following theorem:

\begin{theorem}[$\RCA_0$]\label[theorem]{thm:gr1-bound}
There exists a primitive recursive function $f_{\GR^1} : \NN \to \NN$ such that for every~$n \in \NN$, every $\bbomega^{f_{\GR^1}(n)}$-large$(\theta)$ set is $\GR^1$-$\bbomega^n$-large$(\theta)$. 
\end{theorem}

The core of the argument lies in the following technical proposition, which state sufficient largeness conditions for \qt{grouping} blocks of solutions into a larger solution.

\begin{proposition}\label[proposition]{prop:GR1-blocks}
    Let $n \in \NN$, let $r : \NN \to \NN$ be some non-decreasing function, and suppose that the following induction hypothesis holds:
    \begin{quote}There exists some $a \in \NN$ such that every $(x \mapsto x^{x^x})$-sparse $\bbomega^a$-large$(\theta)$ set $X$ is $\GR^1$-$\bbomega^n\cdot r(\min X)$-large$(\theta)$. 
    \end{quote}
    Let $k, \ell \in \NN$, and $X_0 < \dots < X_{k-1}$ be $\theta$-apart and $\bbomega^{f_{\GR^0}(f_{\GR^0}(a))}$-large$(\theta)$ sets such that $\ell^k < \min X_0$. Let $X = X_0 \cup \dots \cup X_{k-1}$ and let $w = w_0 \cdots  w_{k-1}$ be an $X$-variable word over an alphabet $A$ with $|A| < \min X_0$.

    Then, for every coloring $f : \USub^{1,=}_A(w) \to \ell$, there exist $\bbomega^n\cdot r(\min X_i)$-large$(\theta)$ subsets $W_i \subseteq X_i$ and some $W_0 \cup \dots \cup W_{k-1}$-variable word $v \in \USub_A^=(w)$ such that the color of any $1$-variable word $u \in \USub_A^{1,=}(v)$ only depends on which $X_i$ the first occurrence of the variable belongs.
\end{proposition}

\begin{proof}
We say that some 1-variable word~$u \in \USub^{1,=}_A(w)$ is of \emph{type~$i$} (for some $i < k$) if the first occurrence of the variable belongs to~$X_i$. Given a set~$Y \subseteq X$, the 1-variable words of type~$i$ are \emph{$Y$-invariant} if their color do not depend on the valuations of the variables $\{x_n : n \in Y\}$.
\smallskip

\noindent
\textbf{Proof outline:}
The construction will follow $3k$ steps, in the following order:
\begin{center}
Step $(k-1)$.a, Step $(k-1)$.b, Step $(k-1)$.c, $\dots$, Step 0.a, Step 0.b, Step 0.c.
\end{center}
Let $i <k$, and assume $W_j \subseteq X_j$ have been defined for $j \in \{i+1, \dots, k-1\}$.
\begin{itemize}
    \item In Step i.a, we are going to find some subset $Y_i \subseteq X_i$ and some $X_0 \cup \dots \cup X_{i-1} \cup Y_i \cup W_{i+1} \cup \dots \cup W_{k-1}$-variable word $w^{i,a} \in \USub_A^=(w^{i+1,c})$ (with the convention that $w^{k,c} = w$) such that the $1$-variable words of type $< i$ are $Y_i$-invariant.
    \item In Step i.b, we will find some subset $Z_i \subseteq Y_i$ and some $X_0 \cup \dots \cup X_{i-1} \cup Z_i \cup W_{i+1} \cup \dots \cup W_{k-1}$-variable word $w^{i,b} \in \USub_A^=(w^{i,a})$ such that the $1$-variable words of type $>i$ are $Z_i$-invariant. 
    \item In Step i.c, we will find some subset $W_i \subseteq Z_i$ and some $X_0 \cup \dots \cup X_{i-1} \cup W_i \cup W_{i+1} \cup \dots \cup W_{k-1}$-variable word $w^{i,c} \in \USub_A^=(w^{i,b})$ such that the $1$-variable words of type $i$ are $W_i$-invariant.
\end{itemize}
By induction, at the end of Step i.c, for every 1-variable word $u \in \USub_A^{1,=}(w^{i,c})$,
\begin{itemize}
    \item if the type of $u$ is $< i$, then it is $W_i \cup \dots \cup W_{k-1}$-invariant;
    \item if the type of $u$ is $\geq i$, then $f(u)$ only depends on the valuations of~$X_0 \cup \dots \cup X_{i-1}$ and on its type.
\end{itemize}
In particular, assuming that each $W_i$ is $\bbomega^n \cdot r(\min X_i)$-large$(\theta)$, at the end of Step 0.c, $w^{0,c} \in \USub^=_A(w)$ is a $W_0 \cup \dots \cup W_{k-1}$-variable word satisfying the conclusion of \Cref{prop:GR1-blocks}.
\smallskip

\noindent
\textbf{Formal proof:}
Consider the extended alphabets 
\begin{align*}
    B &:= A \sqcup \{x_p : p \in X_0 \cup \dots \cup X_{i-1}\}\\
    C &:= A \sqcup \{x_p : p \in X_0 \cup \dots \cup X_{i-1} \cup X_i\}
\end{align*}

\noindent
\textbf{Step i.a}. 
Assume $w^{i+1,c} = w^- \cdot w_i \cdot w^+$ has already been defined, where $w^-$ is an $X_0 \cup \cdots \cup X_{i-1}$-variable word over~$A$, $w_i$ is an $X_i$-variable word over~$B$ and $w^+$ is an $W_{i+1} \cup \dots \cup W_{k-1}$-variable word over~$C$. By construction, we have $w^- = w_0 \cdots w_{i-1}$.

\noindent
For every $u^- \in \USub_A^{1,=}(w^-)$, let $$g_{u^-} : \USub_B^{0,=}(w_i) \to \ell$$ be defined by
$$g_{u^-}(u_i) = f(u^- \cdot u_i[x_p \mapsto u^-(p)] \cdot u^+[x_p \mapsto (u^- \cdot u_i)(p)])$$
for any $u^+ \in \USub_{C}^{0,=}(w^+)$. By Steps j.a for $j > i$, the coloring $g_{u^-}$ is well-defined.
Finally, let $$g : \USub_B^{0,=}(w_i) \to \ell^{|\USub_A^{1,=}(w^-)|}$$ be defined by $$g(u_i) = \langle g_{u^-}(u_i) : u^- \in \USub_A^{1,=}(w^-)\rangle$$

We have $\min X_i > \ell^{|\USub_A^{1,=}(w^-)|}$. Indeed, in the case where $i = 0$ we have $\ell^{|\USub_A^{1,=}(w^-)|} = \ell$, and in the case where $i \neq 0$, we have $\ell^{|\USub_A^{1,=}(w^-)|} \leq \ell^{(\max X_{i-1})^{\max X_{i-1}}} < \min X_i$ by $(x \mapsto x^{x^x})$-sparsity of $X$. Hence, by definition of $f_{\GR^0}$, there exist some  $\bbomega^{f_{\GR^0}(a)}$-large$(\theta)$ subset $Y_i \subseteq X_i$ and some $Y_i$-variable word $w_i^{i,a} \in \USub_B^=(w^i)$ such that $\USub_B^{0,=}(w^{i,a}_i)$ is $g$-monochromatic. By letting $w^{i,a} := w^- \cdot w_i^{i,a} \cdot w^+$, we have $w^{i,a} \in \USub_A^=(w^{i+1,c})$ and all the $1$-variable words of type $<i$ are $Y_i$-invariant.
\smallskip

\noindent
\textbf{Step i.b}. Assume $w^{i,a} = w^- \cdot w^{i,a}_i \cdot w^+$ have already been defined. By construction, we have $w^- = w_0 \cdots w_{i-1}$.

\noindent
For every $u^- \in \USub_A^{0,=}(w^-)$ and $j \in \{i+1, \dots, k-1\}$, let $$h_{u^-,j} : \USub_{B}^{0,=}(w_i^{i,a}) \to \ell$$
be defined by $$h_{u^-,j}(u_i) = f(u^- \cdot u_i[x_p \mapsto u^-(p)] \cdot u^+[x_p \mapsto (u^- \cdot u_i)(p)])$$

for any $u^+ \in \USub_{C}^{1,=}(w^+)$ of type $j$. By Steps
j.a, j.b and j.c, the coloring $h_{u^-,j}$ is well-defined.
Finally, let $$h :  \USub_{B}^{0,=}(w_i^{i,a}) \to \ell^{|\USub_A^{0,=}(w^-)| \times (k-i-1)}$$
be defined by $$h(u_i) = \langle h_{u^-,j}(u_i) : j \in \{i+1, \dots, k-1\} \textit{ and } u^- \in \USub_A^{0,=}(w^-) \textit{ of type $j$} \rangle$$

By $(x \mapsto x^{x^x})$-sparsity of $X$ and from the fact that $\ell^k < \min X_0$, we have $\min Y_i > \ell^{|\USub_A^{1,=}(w^-)| \times k}$, hence, by definition of $f_{\GR^0}$, there exists an  $\bbomega^{a}$-large$(\theta)$ subset $Z_i \subseteq Y_i$ and some $Z_i$-variable word $w_i^{i,b} \in \USub_B^=(w^{i,a}_i)$ such that $\USub_{B}^{0,=}(w^{i,b})$ is $h$-monochromatic. By letting $w^{i,b} := w^- \cdot w_i^{i,b} \cdot w^+$, we have $w^{i,b} \in \USub_A^=(w^{i,a})$ and all the $1$-variable words of type $>i$ are $Z_i$-invariant.
\smallskip

\noindent
\textbf{Step i.c}. Assume $w^{i,b} = w^- \cdot w^{i,b}_{i} \cdot w^+$ have already been defined. By construction, we have $w^- = w_0 \cdots w_{i-1}$.

For every $u^- \in \USub_A^{0,=}(w^-)$, let $$t_{u^-} : \USub_{B}^{1,=}(w_i^{i,b}) \to \ell$$ be defined by $$t_{u^-}(u_i) = f(u^- \cdot u_i[x_p \mapsto u^-(p)] \cdot u^+[x_p \mapsto (u^- \cdot u_i)(p)]))$$
for any $u^+ \in \USub_{C}^{0,=}(w^+)$. By Step j.a for $j > i$, the coloring $t_{u^-}$ is well-defined.
Finally, let 
$$t : \USub_{B}^{1,=}(w_i^{i,b}) \to \ell^{|\USub_A^{0,=}(w^-)|}$$
be defined by $$t(u_i) = \langle t_{u^-}(u_i) : u^- \in \USub_A^{0,=}(w^-) \rangle$$
By $(x \mapsto x^{x^x})$-sparsity of $X$, we have $\min Z_i > \ell^{|\USub_A^{0,=}(w^-)|}$, hence, by assumption, there exists some $\bbomega^n\cdot r(\min Z_i)$-large$(\theta)$ subset $W_i \subseteq Z_i$ (and therefore $\bbomega^n\cdot r(\min X_i)$-large$(\theta)$ as $r$ is non-decreasing) and some $W_i$-variable word $w_i^{i,c} \in \USub_B^=(w^{i,b})$ such that $\USub_{B}^{1,=}(w^{i,c})$ is $t$-monochromatic. By letting $w^{i,c} := w^- \cdot w_i^{i,c} \cdot w^+$, we have $w^{i,c} \in \USub_A^=(w^{i,b})$ and all the $1$-variable words of type $i$ are $W_i$-invariant. \\

Finally, $v := w^{0,c}$ has the desired properties.
\end{proof}

We are now ready to prove the main theorem of our section:

\begin{reptheorem}{thm:gr1-bound}[$\RCA_0$]
    There exists a primitive recursive function $f_{\GR^1} : \NN \to \NN$ such that for every $n  \in \NN$, every $\bbomega^{f_{\GR^1}(n)}$-large$(\theta)$ set is $\GR^1$-$\bbomega^n$-large$(\theta)$.
\end{reptheorem}
\begin{proof}
Consider the sequences $(b_n)_{n \in \NN}$ and $(c_n)_{n \in \NN}$ defined by mutual induction as follows: 
$$b_0 = 0, c_{n} = f_{\GR^0}(f_{\GR^0}(b_n)) + 2 \text{ and }  b_{n+1} = f_{\GR^0}(f_{\GR^0}(c_{n})) + 1.$$ 
Then, let $f_{\GR^1}(n) = 2b_n + n_0$ with $n_0$ given by \Cref{lem:largeness-to-sparsity} is such that every $\bbomega^{f_{\GR^1}(n)}$-large$(\theta)$ set contains a $(x \mapsto x^{x^x})$-sparse $\bbomega^{b_n}$-large$(\theta)$ subset.

We prove by mutual induction on the definition of $(b_n)$ and $(c_n)$ that 
\begin{itemize}
    \item[(1)] every $(x \mapsto x^{x^x})$-sparse $\bbomega^{b_n}$-large$(\theta)$ set $X$ is $\GR^1$-$\bbomega^n$-large$(\theta)$; and that
    \item[(2)] every $(x \mapsto x^{x^x})$-sparse $\bbomega^{c_n}$-large$(\theta)$ set $X$ is $\GR^1$-$\bbomega^n\cdot \min X$-large$(\theta)$.
\end{itemize}
Every $\bbomega^{b_0}$-large$(\theta)$ set $X$ is $\GR^1$-$\bbomega^0$-large$(\theta)$, hence we have the base case.

Assume that, for some $n \in \NN$, every $(x \mapsto x^{x^x})$-sparse and $\bbomega^{b_n}$-large$(\theta)$ set is $\GR^1$-$\bbomega^n$-large$(\theta)$. Let $X$ be $(x \mapsto x^{x^x})$-sparse and $\bbomega^{c_n}$-large$(\theta)$ and let $f : \USub^{1,=}_A(w) \to \ell$ be a coloring for some $X$-variable word $w$, some $\ell < \min X$ and some alphabet $A$ with $|A| < \min X$. As $c_{n} = f_{\GR^0}(f_{\GR^0}(b_n)) + 2$, $\hat X := X \setminus \{\min X\}$ is $\bbomega^{f_{\GR^0}(f_{\GR^0}(b_n)) + 1}$-large$(\theta)$. Then, by $(x \mapsto x^{x^x})$-sparsity, $$\min \hat X > \min X \times \ell$$ and there exist subsets $X_0 < \dots < X_{\min X \times \ell -1}$ of $\hat X$ that are pairwise $\theta$-apart and $\bbomega^{f_{\GR^0}(f_{\GR^0}(b_n))}$-large$(\theta)$. We also have $\min X_0 > \ell^{\min X \times \ell}$, hence, by \Cref{prop:GR1-blocks} with $r$ being the constant function 1, there exists some $\bbomega^n$-large$(\theta)$ subsets $W_i \subseteq X_i$ for every $i < \min X \times \ell$ and some $W_0 \cup \dots \cup W_{\min X \times \ell - 1}$-variable word $v \in \USub_A^=(w)$ such that the color of any $1$-variable word $u \in \USub_A^{1,=}(v)$ only depends on its type (i.e., in which $W_i$ does the first occurrence of the variable belong) and let $c_i$ be that color. By the finite pigeonhole principle, there exists some color $c < \ell$ such that $c_i = c$ for at least $\min X$ indices $i$. Let $i_0, \dots, i_{\min X - 1}$ be those indices, and consider the $\bbomega^{n}\cdot (\min X)$-large$(\theta)$ set $W := W_{i_0}\cup \dots \cup W_{i_{\min X -1}}$. Then, by construction, any $1$-variable word $u \in \USub_A^{1,=}(v)$ whose first occurrence of the variable belongs to $W$ will have the same color. Hence, $X$ is $\GR^1$-$\bbomega^{n}\cdot (\min X)$-large$(\theta)$.

Assume that, for some $n \in \NN$, every $(x \mapsto x^{x^x})$-sparse and $\bbomega^{c_n}$-large$(\theta)$ set $Y$ is $\GR^1$-$\bbomega^n\cdot(\min Y)$-large$(\theta)$. Let $X$ be $(x \mapsto x^{x^x})$-sparse and $\bbomega^{b_{n+1}}$-large$(\theta)$ and let $f : \USub^{1,=}_A(w) \to \ell$ be a coloring for some $X$-variable word $w$, some $\ell < \min X$ and some alphabet $A$ with $|A| < \min X$. Let $X_0 < \dots < X_{\min X - 1}$ be pairwise $\theta$-apart and $\bbomega^{f_{\GR^0}(f_{\GR^0}(c_n))}$-large$(\theta)$ subsets of $X \setminus \{\min X\}$. By $(x\mapsto x^{x^x})$-sparsity of $X$, we have $\ell^{\min X} < \min X_0$, hence, by \Cref{prop:GR1-blocks} with $r$ being the identity function, there exists $\bbomega^n \cdot (\min X_i)$-large$(\theta)$ subsets $W_i \subseteq X_i$ for every $i < \min X$ and some $X_0 \cup \dots \cup X_{\min X - 1}$-variable word $v \in \USub_A^=(w)$ such that the color of any $1$-variable word $u \in \USub_A^{1,=}(v)$ only depends on its type and let $c_i$ be the color of the $1$-variable words of type $i$. By the finite pigeonhole principle, there exist two indices $i_0 < i_1$ such that $c_{i_0} = c_{i_1}$. Let $x \in W_{i_0}$ and consider the set $W := \{x\} \cup W_{i_1}$. As $x < \min X_{i_1}$, $W$ is $\bbomega^{n+1}$-large$(\theta)$. Then, by construction, any $1$-variable word $u \in \USub_A^{1,=}(v)$ whose first occurrence of the variable belongs to $W$ will have the same color. Hence, $X$ is $\GR^1$-$\bbomega^{n+1}$-large$(\theta)$.
\end{proof}

\section{Carlson-Simpson lemma}\label[section]{sec:largeness-csl1}

The goal of this section is to prove the following theorem, where a set~$X$ is $\CSL^n_\ell$-$\alpha$-large$(\theta)$ if it is $\CSL^n_\ell(k)$-$\alpha$-large$(\theta)$ for every~$k < \min X$.

\begin{theorem}[$\RCA_0$]\label[theorem]{thm:csl12-bound}
There exists a primitive recursive function $f_{\CSL^1_2} : \NN \to \NN$ such that for every~$n \in \NN$, every $\bbomega^{f_{\CSL^1_2}(n)}$-large$(\theta)$ set is $\CSL^1_2$-$\bbomega^n$-large$(\theta)$. 
\end{theorem}

By iterating \Cref{thm:csl12-bound}, one can obtain bounds for $\CSL^1_\ell$-$\bbomega^n$-largeness$(\theta)$ for every~$\ell \in \NN$. On the other hand, there is no bound for $\CSL^1$-$\bbomega^n$-largeness$(\theta)$ in general by \Cref{prop:csl1-bsig3}. Indeed, the existence of such a bound would yield that $\RCA_0 + \forall \ell, k \CSL^1_\ell(k)$ is $\forall \Pi^0_4$-conservative over~$\RCA_0 + \BSig_2$, but $\RCA_0 \vdash \forall \ell \CSL^1_\ell(2) \to \BSig_3$, and $\BSig_3$ is not even $\Pi_1$-conservative over~$\RCA_0 + \BSig_2$.

However, one can decompose the statement $\forall \ell, k \CSL^1_\ell(k)$ into its \qt{variable word part} and its \qt{Ramsey part}, as follows: Consider the projection function~$\pi$, which maps any $n$-variable word~$w$ over an alphabet~$A$ to the $(n+1)$-tuple $\{p_0 < \dots < p_n \}$ where $p_i$ is the position of the first occurence of the $i$th variable kind, and $p_n = |w|$. Note that $\pi(w) = \set(w) \cup \{|w|\}$. The following statement is a consequence of the Carlson-Simpson lemma:

\begin{statement}[Level $\CSL$]
Fix $n \geq 0$ and~$\ell \geq 1$. For every finite alphabet~$A$ and every coloring $f : A^{<\omega, n} \to \ell$, there is an infinite $\omega$-variable word~$W$ such that for every~$u_0, u_1 \in \USub^{n,\star}_A(W)$, if $\pi(u_0) = \pi(u_1)$, then $f(u_0) = f(u_1)$.
\end{statement}

We denote by $\RFCSL^n_\ell(k)$ the Level Carlson-Simpson lemma for $\ell$-colorings of $n$-variable words over finite alphabets of size~$k$.

\begin{lemma}[Liu and Patey~\cite{liu2026reverse}]\label[lemma]{lem:csl-rfcsl-rt}
$\RCA_0 \vdash \forall n, k, \ell(\CSL^n_\ell(k) \leftrightarrow \RFCSL^n_\ell(k) \wedge \RT^{n+1}_\ell)$.
\end{lemma}

We shall actually see that the whole induction strength of the statement $\CSL^1$ is contained in its Ramsey part, namely, $\RT^2$, by proving that its level version is $\forall \Pi^0_4$-conservative over~$\RCA_0 + \BSig_2$. One could define $\RFCSL^1$-$\bbomega^n$-largeness$(\theta)$ and prove the existence of a primitive recursive function $f_{\RFCSL^1} : \NN \to \NN$ such that every $\bbomega^{f_{\RFCSL^1}(n)}$-large$(\theta)$ set is $\RFCSL^1$-$\bbomega^n$-large$(\theta)$. However, as we shall see in the next section, neither $\CSL^1_2$-$\bbomega^n$-largeness$(\theta)$, nor $\RFCSL^1$-$\bbomega^n$-largeness$(\theta)$ provides a good invariant for proving the existence of such bounds. We now define a notion in between $\CSL^1$ and $\RFCSL^1$ which makes only sense in the finitary version of the statement, for $\bbomega^n$-largeness$(\theta)$. Given an $n$-variable word~$w$ and $i < n$, we write $\pi_i(w)$ for the position of its first occurrence of~$x_i$, and $\pi_n(w)$ for $|w|$.

\begin{definition}
Let $X \finsub \NN$ be a set, $w$ be an $X$-variable word over an alphabet~$A$,
and $f : \USub^{1,\star}_A(w) \to \ell$ be a coloring. We define $f$-$\alpha$-block-homogeneity$(\theta)$ for some color $\sigma$ inductively on~$n$ as follows:
\begin{itemize}
    \item $X$ is $f$-$\bbomega^0$-block-homogeneous$(\theta)$ for color~$\langle \rangle$ if $X \neq \emptyset$.
    \item $X$ is $f$-$\bbomega^n\cdot k$-block-homogeneous$(\theta)$ for color $\sigma \cdot c$ with~$c < \ell$ if $X = X_0 \cup \dots \cup X_{k-1}$ for some there are pairwise $\theta$-apart subsets $X_0 < \dots < X_{k-1}$ such that 
    \begin{itemize}
        \item Each $X_i$ is $\bbomega^n$-block-homogeneous$(\theta)$ for color~$\sigma$;
        \item For every~$u \in \USub^{1,\star}_A(w)$ with $\pi_0(u) \in X_i$ and $\pi_1(u) \in X_j$, if $i < j$ then $f(u) = c$.
    \end{itemize}    
    \item $X$ is $f$-$\bbomega^{n+1}$-block-homogeneous$(\theta)$ for color $\sigma \cdot c$ with~$c < \ell$ if $X \setminus \{\min X\}$ is $f$-$\bbomega^n\cdot (\min X)$-block-homogeneous$(\theta)$ for color $\sigma$ and if $f(u) = c$ for every~$u \in \USub^{1,\star}_A(w)$ with $\pi_0(u) = \min X$.
\end{itemize}
\end{definition}

Note that if $X$ is $f$-$\bbomega^n$-block-homogeneous$(\theta)$ for some color~$\sigma$, then it is a solution to the Level 1-dimensional Carlson-Simpson lemma. The main advantage of this definition is that there are only $\ell^{2n}$ many possible colors of block-homogeneity, while for $\RFCSL^1$, there would be $\ell^{|X|}$ possible colors. This difference will be used in an essential way in the next section. 

\begin{definition}
A set~$X$ is $\BlockCSL{\alpha}$ if for every $X$-variable word~$w$ over an alphabet~$A$ and every coloring $f : \USub^{1,\star}_A(w) \to \ell$ with $|A|, \ell < \min X$, there is a subset~$Y \subseteq X$ and a $Y$-variable word $v \in \USub_A(w)$ such that $Y$ is $f^v$-$\alpha$-block-homogeneous$(\theta)$ for some color, where $f^v$ is the restriction of $f$ to $\USub^{1,\star}_A(v)$.
\end{definition}

The goal is therefore to prove the following theorem:

\begin{theorem}[$\RCA_0$]\label[theorem]{thm:block-homogeneous-bound}
There exists a primitive recursive function $f_{\BCSL^1} : \NN \to \NN$ such that for every~$n \in \NN$, every $\bbomega^{f_{\BCSL^1}(n)}$-large$(\theta)$ set is $\BlockCSL{\bbomega^n}$.
\end{theorem}

In particular, any $\bbomega^{f_{\BCSL^1}(n)}$-large$(\theta)$ set is $\RFCSL^1$-$\bbomega^n$-large$(\theta)$.
As for the Graham-Rothschild theorem, the core of the argument lies in the following proposition which combines blocks of solutions to a larger solution:

\begin{proposition}\label[proposition]{prop:BCSL1-blocks}
    Let $n \in \NN$, let $r : \NN \to \NN$ be a non-decreasing function, and suppose that the following induction hypothesis holds:
    \begin{quote}
    There exists some $a \in \NN$ such that every $(x \mapsto x^{x^x})$-sparse $\bbomega^a$-large$(\theta)$ set $X$ is $\BlockCSL{\bbomega^n \cdot r(\min X)}$.
    \end{quote}
    Let $k, \ell \in \NN$, and $X_0 < \dots < X_{k-1}$ be $\theta$-apart and $\bbomega^{f_{\CSL^0}(f_{\GR^0}(f_{\GR^0}(f_{\GR^1}(a))))}$-large$(\theta)$ sets such that $\ell^{(2n+1)k^2} < \min X_0$. Let $X = X_0 \cup \dots \cup X_{k-1}$ and let $w = w_0 \dots  w_{k-1}$ be an $X$-variable word over an alphabet $A$ of size smaller than $\min X_0$.  
    
    Then, for every coloring $f : \USub^{1,\star}_A(w) \to \ell$, there exist $\bbomega^n\cdot r(\min X_i)$-large$(\theta)$ subsets $W_i \subseteq X_i$ and some $W_0 \cup \cdots \cup W_{k-1}$-variable word $v \in \USub^=_A(w)$ such that:
    \begin{itemize}
        \item[(1)] for every $i < k$, there exists a color $\sigma_i$, such that, for every $W_i$-variable word $u \in \USub_A^{=}(v)$, $W_i$ is $f^u$-$\bbomega^n\cdot r(\min X_i)$-block-homogeneous$(\theta)$ for $\sigma_i$, where $f^u$ is the restriction of $f$ to $\USub^{1,\star}_A(u)$.
        \item[(2)] the color of every $1$-variable word $u \in \USub_A^{1,\star}(v)$ only depends on the pair $(i_0,i_1)$ such that $\pi_0(u) \in W_{i_0}$ and $\pi_1(u) \in W_{i_1}$ when $i_0 \neq i_1$.
    \end{itemize}
\end{proposition}

\begin{proof}
We say that some 1-variable word~$u \in \USub^{1,\star}_A(w)$ is of \emph{type~$(i_0,i_1)$} (for some $i_0 \leq i_1$) if $\pi_0(u) \in X_{i_0}$ and $\pi_1(u) \in X_{i_1}$. Given a set~$Y \subseteq X$, the 1-variable words of type~$(i_0,i_1)$ are \emph{$Y$-invariant} if their color do not depend on the valuations of the variables $\{x_n : n \in Y\}$.
\smallskip

\noindent
\textbf{Proof outline:}
The construction will follow $5k$ steps, in the following order:
\begin{center}
Step $(k-1)$.a, Step $(k-1)$.b, Step $(k-1)$.c, Step $(k-1)$.d, Step $(k-1)$.e, $\dots$, Step 0.a, Step 0.b, Step 0.c, Step 0.d, Step 0.e.
\end{center}
These steps will ensure that the $1$-variable words $u \in \USub_A^{1,\star}(v)$ of any type $(i_0,i_1)$ are $W_i$-invariant for every $i < k$, except in the case where $i_0 = i_1 = i$, where we will find some color $\sigma_i$ and ensure the block-homogeneity of $W_i$ for that color.
\smallskip

\noindent
Let $i <k$, and assume $W_j \subseteq X_j$ have been defined for $j \in \{i+1, \dots, k-1\}$.
\begin{itemize}
    \item In Step i.a, we are going to find some subset $Y_{i,a} \subseteq X_i$ and some $X_0 \cup \dots \cup X_{i-1} \cup Y_{i,a} \cup W_{i+1} \cup \dots \cup W_{k-1}$-variable word $w^{i,a} \in \USub_A^=(w^{i+1,e})$ (or $w^{i,a} \in \USub_A^=(w)$ in the case where $i = k-1$) such that the $1$-variable words of type $(i_0,i)$ with $i_0 < i$ are $Y_{i,a}$-invariant.
    \item In Step i.b, we are going to find some subset $Y_{i,b} \subseteq Y_{i,a}$ and some $X_0 \cup \dots \cup X_{i-1} \cup Y_{i,b} \cup W_{i+1} \cup \dots \cup W_{k-1}$-variable word $w^{i,b} \in \USub_A^=(w^{i,a})$ such that the $1$-variable words of type $(i_0,i_1)$ with $i_0,i_1 > i$ are $Y_{i,b}$-invariant.
    \item In Step i.c, we are going to find some subset $Y_{i,c} \subseteq Y_{i,b}$ and some $X_0 \cup \dots \cup X_{i-1} \cup Y_{i,c} \cup W_{i+1} \cup \dots \cup W_{k-1}$-variable word $w^{i,c} \in \USub_A^=(w^{i,b})$ such that the $1$-variable words of type $(i_0,i_1)$ with $i_0 < i < i_1$ are $Y_{i,c}$-invariant.
    \item In Step i.d, we are going to find some subset $Y_{i,d} \subseteq Y_{i,c}$ and some $X_0 \cup \dots \cup X_{i-1} \cup Y_{i,d} \cup W_{i+1} \cup \dots \cup W_{k-1}$-variable word $w^{i,d} \in \USub_A^=(w^{i,c})$ such that the $1$-variable words of type $(i,i_1)$ with $i < i_1$ are $Y_{i,d}$-invariant.
    \item In Step i.e, we are going to find some subset $W_i \subseteq Y_{i,d}$ and some $X_0 \cup \dots \cup X_{i-1} \cup W_i \cup W_{i+1} \cup \dots \cup W_{k-1}$-variable word $w^{i,e} \in \USub_A^=(w^{i,d})$ such that, for every $W_i$-variable word $u \in \USub_A^{=}(w^{i,e})$, $W_i$ is $f^u$-$\bbomega^n\cdot r(\min X_i)$-block-homogeneous$(\theta)$ for some $\sigma_i \in \ell^{2n+1}$, where $f^u$ is the restriction of $f$ to $\USub^{1,\star}_A(u)$.

\end{itemize}  

Notice that we already have that the $1$-variable words of type $(i_0,i_1)$ are $X_i$ invariant when $i_0,i_1 < i$ as the $1$-variable words of type~$(i_0,i_1)$ are of length $< \min X_i$.
\smallskip

\noindent
\textbf{Formal proof:} Consider the extended alphabets 
\begin{align*}
    B &:= A \sqcup \{x_p : p \in X_0 \cup \dots \cup X_{i-1}\}\\
    C &:= A \sqcup \{x_p : p \in X_0 \cup \dots \cup X_{i-1} \cup X_i\}
\end{align*}
At each step of the construction, we will be defining a coloring as follows:
\smallskip

\noindent
\textbf{Step i.a}. 
Assume $w^{i+1,e} = w^- \cdot w_i \cdot w^+$ has already been defined, where $w^-$ is an $X_0 \cup \cdots \cup X_{i-1}$-variable word over~$A$, $w_i$ is an $X_i$-variable word over~$B$ and $w^+$ is an $W_{i+1} \cup \dots \cup W_{k-1}$-variable word over~$C$. By convention,  $w^{i+1,e} = w$ if $i = k-1$. By construction, we have $w^- = w_0 \cdots w_{i-1}$.

For every $u^- \in \USub_A^{1,=}(w^-)$, let $$g_{i.a}^{u^-} : \USub_{B}^{0,\star}(w_i) \to \ell$$ be defined by
$$g_{i.a}^{u^-}(u_i) = f(u^- \cdot u_i[x_p \mapsto u^-(p)])$$

Then, let $$g_{i.a} : \USub_{B}^{0,\star}(w_i) \to \ell^{|\USub_A^{1,=}(w^-)|}$$ be defined by $$g_{i.a}(u_i) = \langle g_{i.a}^{u^-}(u_i) : u^- \in \USub_A^{1,=}(w^-)\rangle$$
\smallskip

\noindent
\textbf{Step i.b}. Assume $w^{i,a} = w^- \cdot w^{i,a}_i \cdot w^+$ have already been defined. By construction, we have $w^- = w_0 \cdots w_{i-1}$.

For every $u^- \in \USub_A^{0,=}(w^-)$ and $i_0,i_1 \in \NN$ with $i < i_0 < i_1 < k$, let $$g_{i.b}^{u^-,i_0,i_1} : \USub_{B}^{0,=}(w^{i,a}_i) \to \ell$$ be defined by $$g_{i.b}^{u^-, i_0, i_1}(u_i) = f(u^- \cdot u_i[x_p \mapsto u^-(p)] \cdot u^+[x_p \mapsto (u^- \cdot u_i)(p)])$$ for any $u^+ \in \USub_C^{1,\star}(w^+)$ of type $(i_0, i_1)$. By Steps $j$.a, $j$.b, $j$.c and $j$.d for $j > i$, the coloring $g_{i.b}^{u^-, i_0, i_1}$ is well-defined.

For every $u^- \in \USub_A^{0,=}(w^-)$ and $i_0 \in \NN$ with $i < i_0 < k$, let $$g_{i.b}^{u^-,i_0} : \USub_{B}^{0,=}(w^{i,a}_i) \to \ell^{2n+1}$$ be defined by $g_{i.b}^{u^-, i_0}(u_i) = \sigma$
where $f^{u^-, u_i}$ is the restriction of $f$ to $\USub^{1,\star}_A(u^- \cdot u_i[x_p \mapsto u^-(p)])$ and $\sigma$ is the color of $f^{u^-,u_i}$-$\bbomega^n \cdot r(\min X_{i_0})$-block-homogeneity$(\theta)$ for $W_{i_0}$.
By Steps $j$.e for $j > i$, the coloring $g_{i.b}^{u^-, i_0}$ is well-defined.

Finally, let 
$$g_{i.b} : \USub_{B}^{0,=}(w^{i,a}_i) \to \ell^{|\USub_A^{0,=}(w^-)|\times \left((2n+1)(k-i-1) + \frac{(k-i-1)(k-i-2)}{2}\right)}$$
be defined by $$g_{i.b}(u_i) = \left\langle \begin{array}{ll}
    g_{i.b}^{u^-, i_0, i_1}(u_i) &: i_0 < i_1 \in \{i+1, \dots, k-1\} \textit{ and } u^- \in \USub_A^{0,=}(w^-)\\
    g_{i.b}^{u^-, i_0}(u_i) &: i_0 \in \{i+1, \dots, k-1\} \text{ and } u^- \in \USub_A^{0,=}(w^-)\\
    \end{array}\right\rangle$$


\smallskip

\noindent
\textbf{Step i.c}. Assume $w^{i,b} = w^- \cdot w^{i,b}_i \cdot w^+$ have already been defined. By construction, we have $w^- = w_0 \cdots w_{i-1}$.

For every $u^- \in \USub_A^{1,=}(w^-)$ and $i_1 \in \{i+1, \dots, k-1\}$, let $$g_{i.c}^{u^-,i_1} : \USub_{B}^{0,=}(w^{i,b}_i) \to \ell$$ be defined by
$$g_{i.c}^{u^-,i_1}(u_i) = f(u^- \cdot u_i[x_p \mapsto u^-(p)] \cdot u^+[x_p \mapsto (u^- \cdot u_i)(p)])$$ for any $u^+ \in \USub_C^{0,\star}(w^+)$ with $|u^+| \in W_{i_1}$. Recall that we consider $u^+$ as a located variable word, in which case $|u^+|$ denotes the least upper bound of its domain. By Steps j.a and j.c for $j > i$, the coloring $g_{i.c}^{u^-,i_1}$ is well-defined.

Finally, let $$g_{i.c} : \USub_{B}^{0,=}(w^{i,b}_i) \to \ell^{|\USub_A^{1,=}(w^-)|\times (k-i-1)}$$ be defined by $$g_{i.c}(u_i) = \langle g_{i.c}^{u^-,i_1}(u_i) : i_1 \in \{i+1, \dots, k-1\} \textit{ and } u^- \in \USub_A^{1,=}(w^-) \rangle$$
\smallskip

\noindent
\textbf{Step i.d}. Assume $w^{i,c} = w^- \cdot w^{i,c}_i \cdot w^+$ have already been defined. By construction, we have $w^- = w_0 \cdots w_{i-1}$.

For every $u^- \in \USub_A^{0,=}(w^-)$ and $i_1 \in \{i+1, \dots, k-1\}$, let $$g_{i.d}^{u^-,i_1} : \USub_B^{1,=}(w^{i,c}_i) \to \ell$$ be defined by  
$$g_{i.d}^{u^-,i_1}(u_i) = f(u^- \cdot u_i[x_p \mapsto u^-(p)] \cdot u^+[x_p \mapsto (u^- \cdot u_i)(p)])$$ for any $u^+ \in \USub_{C}^{0,\star}$ with $|u^+| \in W_{i_1}$. By Steps j.a and j.c for $j > i$, the coloring $g_{i.d}^{u^-,i_1}$ is well-defined.

Finally, let $$g_{i.d} : \USub_B^{1,=}(w^{i,c}_i) \to \ell^{|\USub_A^{0,=}(w^-)|\times (k-i-1)}$$ be defined by $$g_{i.d}(u_i) = \langle g_{i.d}^{u^-,i_1}(u_i) : i_1 \in \{i+1, \dots, k-1\} \textit{ and } u^- \in \USub_A^{0,=}(w^-) \rangle$$
\smallskip

\noindent
\textbf{Step i.e}. Assume $w^{i,d} = w^- \cdot w^{i,d}_i \cdot w^+$ have already been defined. By construction, we have $w^- = w_0 \cdots w_{i-1}$. 
For every $u^- \in \USub_A^{0,=}(w^-)$, let $$g_{i.e}^{u^-} : \USub_B^{1,\star}(w^{i,d}_i) \to \ell$$ be defined by 
$$g_{i.e}^{u^-}(u_i) = f(u^- \cdot u_i[x_p \mapsto u^-(p)])$$

Then, let $$g_{i.e} : \USub_B^{1,\star}(w^{i,d}_i) \to \ell^{|\USub_A^{0,=}(w^-)|}$$ be defined by $$g_{i.e}(u_i) = \langle g_{i.e}^{u^-}(u_i) : u^- \in \USub_A^{0,=}(w^-)\rangle$$

\smallskip
In Steps i.a through i.e, we have considered successively: an instance of $\CSL^0$, an instance of $\GR^0$, an instance of $\GR^0$, an instance of $\GR^1$ and an instance of $\CSL^1$. As in the proof of \Cref{prop:GR1-blocks}, we can use the bounds $f_{\GR^0}, f_{\GR^1}$ and $f_{\CSL^0}$ and the definition of $a$, to obtain that it is sufficient for $X_i$ to be $(x \mapsto x^{x^x})$-sparse and $\bbomega^{f_{\CSL^0}(f_{\GR^0}(f_{\GR^0}(f_{\GR^1}(a))))}$-large$(\theta)$ in order to find a sequence of subsets $W_i \subseteq Y_{i,d} \subseteq Y_{i,c} \subseteq Y_{i,b} \subseteq Y_{i,a} \subseteq X_i$, some $X_0 \cup \dots X_{i-1} \cup W_i \cup W_{i+1} \cup \dots \cup W_{k-1}$-variable word $w^{i,e} \in \USub_A^{=}(w)$, and some color $\sigma_i$, so that:

\begin{itemize}
    \item for every $W_i$-variable word $u \in \USub_A^{=}(w^{i,e})$, $W_i$ is $f^u$-$\bbomega^n\cdot r(\min X_i)$-block-homogeneous$(\theta)$ for $\sigma_i$, where $f^u$ is the restriction of $f$ to $\USub^{1,\star}_A(u)$.
    
    \item the $1$-variable word $u \in \USub_A^{1,\star}(w^{i,e})$ of type $(i_0,i_1)$ are $W_i$-invariant when $(i_0,i_1) \neq (i,i)$.
\end{itemize}

The assumption that $\ell^{(2n+1)k^2} < \min X_0$ is used in Step 0.b to ensure that $$\ell^{\left((2n+1)(k-1) + \frac{(k-1)(k-2)}{2}\right)} \leq \ell^{(2n+1)k^2} < \min X_0$$ so that the bound $f_{\GR^0}$ can be used.
Thus, once all the steps have been done, the set $W_0 \cup \dots \cup W_{k-1}$ and the resulting $W_0 \cup \dots \cup W_{k-1}$-variable word $w^{0,e}$ satisfy the desired properties.
\end{proof}

In what follows, for $k, \ell \in \NN$, let $R_{\ell}(k) \in \NN$ denote the smallest bound $b \in \NN$ such that every coloring $f : [b]^2 \to \ell$ admits a subset $H \subseteq b$ of cardinality $k$ satisfying $f([H]^2) = 1$. 
The standard upper bound for the finite Ramsey's theorem for pairs and two colors, obtained by Erd\H{o}s and Szekeres in \cite{erdos1935combinatorial}, states that $$R_2(d) \leq \binom{2d-1}{d-1} \leq 4^{d}.$$ This upper bound generalizes to $R_k(d) \leq k^{kd}$ (see Graham, Rothschild and Spencer~\cite[Section 1.1]{graham2013ramsey}).

\begin{reptheorem}{thm:block-homogeneous-bound}[$\RCA_0$]
    There exists a primitive recursive function $f_{\BCSL^1} : \NN \to \NN$ such that for every~$n \in \NN$, every $\bbomega^{f_{\BCSL^1}(n)}$-large$(\theta)$ set is $\BlockCSL{\bbomega^n}$.
\end{reptheorem}

\begin{proof}
Consider the sequences $(b_n)_{n \in \NN}$ and $(c_n)_{n \in \NN}$ defined by mutual induction as follows: 

$$
\left\{
\begin{aligned}
b_0 &= 0, \\
c_n &= f_{\CSL^0}(f_{\GR^0}(f_{\GR^0}(f_{\GR^1}(b_n)))) + 2, \\
b_{n+1} &= f_{\CSL^0}(f_{\GR^0}(f_{\GR^0}(f_{\GR^1}(c_n)))) + 2
\end{aligned}
\right.
$$
Then, let $f_{\BCSL^1}(n) = 2b_n + n_0$ with $n_0$ given by \Cref{lem:largeness-to-sparsity} is such that every $\bbomega^{f_{\BCSL^1}(n)}$-large$(\theta)$ set contains a $(x \mapsto x^{x^x})$-sparse $\bbomega^{b_n}$-large$(\theta)$ subset. \\

We prove by mutual induction on the definition of $(b_n)$ and $(c_n)$ that 
\begin{itemize}
    \item[(1)] every $(x \mapsto x^{x^x})$-sparse $\bbomega^{b_n}$-large$(\theta)$ set $X$ is $\BCSL^1$-$\bbomega^n$-large$(\theta)$; and that
    \item[(2)] every $(x \mapsto x^{x^x})$-sparse $\bbomega^{c_n}$-large$(\theta)$ set $X$ is $\BCSL^1$-$\bbomega^n\cdot \min X$-large$(\theta)$.
\end{itemize}

Every $\bbomega^{b_0}$-large$(\theta)$ set $X$ is $\BCSL^1$-$\bbomega^0$-large$(\theta)$, hence we have the base case. \\

Assume that every $(x \mapsto x^{x^x})$-sparse $\bbomega^{b_n}$-large$(\theta)$ set is $\BCSL^1$-$\bbomega^n$-large$(\theta)$ and let $X$ be $\bbomega^{c_{n}}$-large$(\theta)$ and $(x \mapsto x^{x^x})$-sparse. Let $w$ be an $X$-variable word over an alphabet $A$ with $|A| < \min X$ and let $f : \USub_A^{1,\star}(w) \to \ell$ be a coloring for some $\ell < \min X$. 

Let $a = \ell^{2n} \times R_{\ell}(\min X)$, we therefore have $a \leq \ell^{2n} \times \ell^{\ell \times \min X}$. As $X$ is $\bbomega^{c_{n}}$-large$(\theta)$, there exists some $\bbomega^{f_{\CSL^0}(f_{\GR^0}(f_{\GR^0}(f_{\GR^1}(b_n)))) + 1}$-large$(\theta)$ subsets $X' < X''$ of $X$. Then, by $(x \mapsto x^{x^x})$-sparsity of $X$ and the fact that $f_{\CSL^0}(f_{\GR^0}(f_{\GR^0}(f_{\GR^1}(b_n)))) + 1 > n$, we get that $\min X'' \geq a$, hence that there exist $X_0 < \dots < X_{a - 1}$ pairwise $\theta$-apart $\bbomega^{f_{\CSL^0}(f_{\GR^0}(f_{\GR^0}(f_{\GR^1}(b_n))))}$-large$(\theta)$ subsets of $X''$. 

We also have $\min X'' >  \ell^{(2n+1)a^2}$, hence, by \Cref{prop:BCSL1-blocks}, there exist $\bbomega^{n}$-large$(\theta)$ subsets $Y_i \subseteq X_i$ and some $Y_0 \cup \dots \cup Y_{a - 1}$-variable word $v \in \USub^=_A(w)$ such that: 

\begin{itemize}
    \item[(1)] for every $i < a$, there exists a color $\sigma_i$, such that, for every $W_i$-variable word $u \in \USub_A^{=}(v)$, $W_i$ is $f^u$-$\bbomega^n$-block-homogeneous$(\theta)$ for $\sigma_i$, where $f^u$ is the restriction of $f$ to $\USub^{1,\star}_A(u)$.
    \item[(2)] the color of every $1$-variable word $u \in \USub_A^{1,\star}(v)$ only depends on the pair $(i_0,i_1)$ such that $\pi_0(u) \in W_{i_0}$ and $\pi_1(u) \in W_{i_1}$ when $i_0 \neq i_1$.
\end{itemize}




By the finite pigeonhole principle, there exists some $\sigma \in \ell^{2n}$ for which $\sigma_i = \sigma$ for at least $R_{\ell}(\min X)$ values of $i < a$. Let $V_{0} < \dots < V_{R_{\ell}(\min X) - 1}$ be the corresponding blocks and pick some $V_{0} \cup \dots \cup V_{R_{\ell}(\min X) - 1}$-variable word $v' \in \USub^=_A(v)$. By our application of \Cref{prop:BCSL1-blocks}, the color of any $u \in \USub_A^{1,\star}(v')$ only depends on the pair $(i_0, i_1)$ such that $\pi_0(u) \in V_{i_0}$ and $\pi_0(u) \in V_{i_1}$ when $i_0 \neq i_1$. Hence, by definition of $R_{\ell}(\min X)$, there exists $\min X$ blocks $K_0 < \dots < K_{\min X - 1}$ among $V_{0}, \dots, V_{R_{\ell}(\min X) - 1}$ and some $c < \ell$ such that $f(u) = c$ for any $u \in \USub_A^{1,\star}(v')$ for which $\pi_0(u) \in K_i$ and $\pi_0(u) \in K_j$ for some $i < j$. Thus, the set $K_0 \cup \dots \cup K_{\min X - 1}$ is $f$-$\bbomega^n\cdot \min X$-block-homogeneous$(\theta)$ for color $\sigma \cdot c$. \\

Assume that every $(x \mapsto x^{x^x})$-sparse $\bbomega^{c_n}$-large$(\theta)$ set is $\BCSL^1$-$\bbomega^n\cdot \min X$-large$(\theta)$ and let $X$ be $\bbomega^{b_{n+1}}$-large$(\theta)$ and $(x \mapsto x^{x^x})$-sparse. Let $w$ be an $X$-variable word over an alphabet $A$ with $|A| < \min X$ and let $f : \USub_A^{1,\star}(w) \to \ell$ be a coloring for some $\ell < \min X$. 

As $X$ is $\bbomega^{b_{n+1}}$-large$(\theta)$ and $(x \mapsto x^{x^x})$-sparse, there exist $X_0 < \dots < X_{\ell^{2n+1}}$ pairwise $\theta$-apart $\bbomega^{f_{\CSL^0}(f_{\GR^0}(f_{\GR^0}(f_{\GR^1}(b_n))))}$-large$(\theta)$ subsets of $X$ by a reasoning similar to the previous case. Moreover, $\ell^{(2n+1)\ell^{4n+2}} < \min X_0$, so by \Cref{prop:BCSL1-blocks}, there exist $\bbomega^{n}\cdot (\min X_i)$-large$(\theta)$ subsets $Y_i \subseteq X_i$ and some $Y_0 \cup \dots \cup Y_{\ell^{2n+1}}$-variable word $v \in \USub^=_A(w)$ such that:

\begin{itemize}
    \item[(1)] for every $i \leq \ell^{2n+1}$, there exists a color $\sigma_i$, such that, for every $W_i$-variable word $u \in \USub_A^{=}(v)$, $W_i$ is $f^u$-$\bbomega^n\cdot \min X_i$-block-homogeneous$(\theta)$ for $\sigma_i$, where $f^u$ is the restriction of $f$ to $\USub^{1,\star}_A(u)$.
    \item[(2)] the color of every $1$-variable word $u \in \USub_A^{1,\star}(v)$ only depends on the pair $(i_0,i_1)$ such that $\pi_0(u) \in W_{i_0}$ and $\pi_1(u) \in W_{i_1}$ when $i_0 \neq i_1$.
\end{itemize}

By the finite pigeonhole principle, there exists some indices $i_0 < i_1 \leq \ell^{2n+1}$ such that $\sigma_{i_0} = \sigma_{i_1}$. Let $x \in Z_{i_0}$, then, as $x \leq \max Z_{i_0} < \min Y_{i_1}$, the set $\{x\} \cup Z_{i_1}$ is $f$-$\bbomega^{n+1}$-block-homogeneous$(\theta)$ for color $\sigma_{i_0} \cdot c$ with $c = f(u)$ for any $u \in \USub_A^{1,\star}(v)$ such that $\pi_0(u) \in Y_{i_0}$ and $\pi_0(u) \in Y_{i_1}$.
\end{proof}

\section{Conservation theorems}\label[section]{sect:conservation}

We now apply the previous closure theorems under $\alpha$-largeness$(\theta)$ to deduce $\forall \Pi^0_4$-conservation of~$\CSL^1_2$ over~$\RCA_0 + \BSig_2$. Such partial conservation results are proven as follows: given a $\forall \Pi^0_4$-sentence~$\varphi$ and a countable model~$\M \models \RCA_0 + \BSig_2 + \neg \varphi$, we will construct another model $\mathcal{N} \models \RCA_0 + \CSL^1_2 + \neg \varphi$. The first-order part of our model~$\mathcal{N}$ will actually be an initial segment of the first-order part of~$\mathcal{M}$.

\begin{definition}[Cut]
    A \emph{cut}\index{cut} in a model $M$ of first-order arithmetic is a nonempty subset $I \subseteq M$ which is closed under successor (if $a \in I$, then $a+1 \in I$), and is an initial segment of~$M$ (if $a \in I$ and $b \leq a$, then $b \in I$).
    A cut $I \subseteq M$ is \emph{proper} if $I \neq M$.
\end{definition}

The model~$\mathcal{N}$ that we shall construct will be fully characterized the cut forming its first-order part. Indeed, its second-order part of~$\mathcal{N}$ will the collection of all its coded sets:

\begin{definition}[Coded sets]\index{set!$M$-coded}
    For $I \subseteq M$ a proper cut, a set $X \subseteq I$ is said to be \emph{$M$-coded} if $X = I \cap \hat X$ for some $M$-finite set $\hat X$. Let $\Cod(M/I)$ be the set of all $M$-coded sets.
\end{definition}

There exists a neat characterization of the cuts~$I \subseteq_e M$ such that $\Cod(M,/I)$ forms a model of~$\WKL_0$. These cuts are called \emph{semi-regular}, a notion related to regularity in set theory:

\begin{definition}[Semi-regular cut \cite{kirby1977initial}]
    A proper cut $I \subseteq M$ is said to be \emph{semi-regular} if for every $M$-finite set $E \subseteq M$ such that $|E| \in I$, $E \cap I$ is bounded in $I$ (i.e., there exists $a \in I$ such that $E \cap I < a$).
\end{definition}

\begin{proposition}[{Scott \cite{scott1962algebras} and Kirby and Paris \cite[Proposition 1]{kirby1977initial}. See also Simpson \cite[Lemma IX.3.11]{simpson2009subsystems}}]\label[proposition]{prop:semi-regular-wkl0}
    Let $M \models \PRA$ and let $I \subseteq M$ be a proper cut. Then, $I$ is semi-regular if and only if $(I, \Cod(M/I)) \models \WKL_0$.   
\end{proposition}

We now turn to the actual proof of $\forall \Pi^0_4$-conservation, from which all our results will follow:

\begin{theorem}\label[theorem]{thm:rfcsl1-pi04-conservation}
$\WKL_0 + \RT^2_2 + \RFCSL^1$ is $\forall \Pi^0_4$-conservative over $\RCA_0 + \BSig_2$.
\end{theorem}

\begin{proof}
Let $\phi(X\uh_z,t,x,y,z)$ be a $\Delta_0^0$-formula and assume that $\RCA_0 + \BSig_2 \not \vdash \forall X \forall t \exists x \forall y \exists z \phi(X\uh_z,t,x,y,z)$. In what follows, we enrich the language with two constant symbols $c$ and $C$, of first and second order, respectively.

By completeness and the Löwenheim-Skolem theorem, there exists some countable model $$\M = (M,S,c^\M, C^\M) \models \RCA_0 + \BSig_2 + \forall x \exists y \forall z \neg \phi(C\uh_z,c,x,y,z)$$
Let $\theta(x,y,z)$ be the $\Delta_0^{0,C,c}$ formula $\neg \phi(C\uh_z,c,x,y,z)$. \\

By \Cref{prop:largeness-bsig2-largeness}, for all standard $n$, there exists some $\omega^n$-large$(\theta)$ subset~$X$ of $M$ such that $c \leq \min X$. So, by compactness, we can assume that $\M$ contains an $\omega^{d}$-large$(\theta)$ set $X$ with $c \leq \min X$ for some non-standard integer $d$. \\

By a standard argument, using \Cref{thm:rt22-bound} for (4) and \Cref{thm:block-homogeneous-bound} for (5), we can consider a decreasing sequence $X = X_0 \supseteq X_1 \supseteq \dots$ such that for every $i < \omega$:

\begin{itemize}
\item[(1)] $X_i$ is $\omega^{d_i}$-large$(\theta)$ for some non-standard $d_i \in M$;
\item[(2)] $\min X_{i+1} > \min X_i$;
\item[(3)] For every $M$-finite set $E$ such that $|E| < \min X_i$, there exists some $j > i$ such that $[\min X_j, \max X_j] \cap E = \emptyset$;
\item[(4)] For every coloring $f : [M]^2 \to 2$ in $\M$, there exists some $j > i$ such that $[X_j]^2$ is $f$-monochromatic;
\item[(5)] For every coloring $f : A^{<M,1} \to \ell$ for some $\ell < \min X_i$ and some $M$-finite alphabet $A$ with $|A| < \min X_i$ with $f \in \M$, there exists some $j > i$ and an $X_j$-variable word~$W \in \M$ such that for every~$u_0, u_1 \in \USub^{1,\star}_A(W)$ for which $\pi(u_0) = \pi(u_1)$, $f(u_0) = f(u_1)$;
\item[(6)] There exists some $j > i$ such that 
$$\forall x < \min X_{j} \exists y < \min X_{j+1} \forall z < \max X_{j+1} \theta(x,y,z)$$

\end{itemize}

Then, consider $I = \sup \{\min X_i | i \in \omega\}$. $I$ is a semi-regular cut of $M$ by $(3)$, therefore $(I, \mathsf{Cod}(M/I)) \models \WKL_0$. The constraint $\min X \geq c$ imposed on largeness$(\theta)$ ensures that $c$ is in $I$. The condition $\min X_{i+1} > \min X_i$ implies that every $X_i \cap I$ is infinite in $I$.

Let $f : [I]^2 \to 2$ be a coloring in $(I, \Cod(M/I))$. There exists some coloring $g : [M]^2 \to 2$ in~$\M$ such that $f = g \cap [I]^2$. By (4), there is some~$j > i$ such that $[X_j]^2$ is $g$-monochromatic. Since $X_j \cap I$ is infinite in~$I$, $(I, \Cod(M/I)) \models \RT^2_2$.

Let $\ell \in I$ and $A$ a finite alphabet with $|A| \in I$, let $f : A^{<I,1} \to \ell$ be a coloring in $(I, \Cod(M/I))$. There exists some coloring $g : A^{<M,1} \to \ell$ in $\M$ such that $f = g \uh I$ and let $i$ such that $\ell, |A| < \min X_i$. By (5), there is some $j > i$ and an $X_j$-variable word~$W \in \M$ such that for every~$u_0, u_1 \in \USub^{1,\star}_A(W)$ for which $\pi(u_0) = \pi(u_1)$, $g(u_0) = g(u_1)$; then $W \uh I$ is an $X_j \cap I$-variable word such that for every~$u_0, u_1 \in \USub^{1,\star}_A(W \uh I)$ for which $\pi(u_0) = \pi(u_1)$, $f(u_0) = f(u_1)$. Since $X_j \cap I$ is infinite in $I$, $W \uh I$ is an $\omega$-variable word in $(I, \Cod(M/I))$, so $(I, \Cod(M/I)) \models \RFCSL^1$.

Finally, $(I, \mathsf{Cod}(M/I),c^{\M},C^{\M}\cap I) \models \forall x \exists y \forall z \theta(x,y,z)$ 
as for every $k \in I$, there exists an index $i \in \omega$ such that $k < \min X_i$ and therefore by (6), there is some~$j > i$ such that $\forall x < k \exists y < \min X_{j+1} \forall z < \max X_{j+1} \ \theta(x,y,z)$, so $(I, \mathsf{Cod}(M/I), c^{\M},C^{\M} \cap I) \models \forall x < k \exists y \forall z \ \theta(x,y,z)$ (since $\max X_{j+1} > I$). 
Hence 
$$(I, \mathsf{Cod}(M/I)) \models \WKL_0 + \RT^2_2 + \RFCSL^1 + \neg \forall X \forall t \exists x \forall y \exists z \phi(X \uh_z,t,x,y,z)$$
\end{proof}

This implies in particular our main theorem:

\begin{repmaintheorem}{thm:csl12-pi04-conservation}
$\WKL_0 + \CSL^1_2$ is $\forall \Pi^0_4$-conservative over~$\RCA_0 + \BSig_2$.
\end{repmaintheorem}
\begin{proof}
Immediate by \Cref{thm:rfcsl1-pi04-conservation} and \Cref{lem:csl-rfcsl-rt}.
\end{proof}

\begin{corollary}\label[corollary]{cor:csl12-pi03-rca}
$\WKL_0 + \CSL^1_2$ is $\forall \Pi^0_3$-conservative over~$\RCA_0$ and $\Pi^0_2$-conservative over~$\mathsf{PRA}$.
\end{corollary}
\begin{proof}
By a parameterized version of the Parsons, Paris and Friedman conservation theorem (see~\cite{hajek1998metamathematics,kaye1991models}), $\RCA_0 + \BSig_2$ is $\forall \Pi^0_3$-conservative over~$\RCA_0$. By Friedman~\cite{friedmancom} (see \cite{simpson2009subsystems}), $\RCA_0$ is $\Pi^0_2$-conservative over~$\mathsf{PRA}$.
\end{proof}

\begin{corollary}[Liu and Patey~\cite{liu2026reverse}]\label[corollary]{cor:csl12-separation-aca}
$\WKL_0 + \CSL^1_2$ does not imply~$\ACA_0$.
\end{corollary}
\begin{proof}
By Simpson~\cite[Corollary VIII.1.7]{simpson2009subsystems}, $\ACA_0$ proves the consistency of~$\RCA_0$, which is a $\Pi^0_1$ statement, while $\RCA_0$ does not by the second G\"odel incompleteness theorem. Thus, by \Cref{cor:csl12-pi03-rca}, $\WKL_0 + \CSL^1_2$ does not imply the consistency of~$\RCA_0$.
\end{proof}

\subsection{Universal triangle-free graph}\label[section]{sec:triangle-free}

We now present an application in Structural Ramsey theory.
Let $\HH_3$ be the universal triangle-free Henson graph, which is the Fraiss\'e limit of the class of all finite triangle-free graphs. Given two structures~$\mathbf F, \mathbf G$, we write ${\mathbf G \choose \mathbf F}$ for the set of all embeddings from~$\mathbf F$ to~$\mathbf G$.
We write $\mathbf F\cong \mathbf G$ iff $\mathbf F$ is isomorphic to $\mathbf G$.
Given integers $k, \ell \in \omega$, structures $\mathbf A, \mathbf B$ and $\mathbf C$ we write $\mathbf C \longrightarrow (\mathbf A)^{\mathbf B}_{k,\ell}$ for the following statement:

\begin{definition}
$\mathbf C \longrightarrow (\mathbf A)^{\mathbf B}_{k,\ell}$: For any coloring $\chi : {\mathbf A \choose \mathbf B} \to k$, there is an embedding $g \in {\mathbf A \choose \mathbf C}$ such that ${g(\mathbf C) \choose \mathbf B}$ uses at most~$\ell$ many colors.
\end{definition}

The \textit{big Ramsey degree} of $\mathbf B$ in $\mathbf A$ is the smallest $L \in \omega \cup \{\omega\}$ such that $\mathbf A \longrightarrow (\mathbf A)^{\mathbf B}_{k,L}$ for every $k \in \omega$, and a infinite structure $\mathbf A$ is said to have \textit{finite big Ramsey degrees} if for all finite substructure of $\mathbf B$, $\mathbf B$ has a finite big Ramsey degree in $\mathbf A$.

Dobrinen~\cite{dobrinen2020ramsey} proved that the triangle-free Henson graph admits finite big Ramsey degrees using a notion of forcing with coding trees. More recently, Hubička~\cite{hubivcka2020big} gave an alternative proof of Dobrinen's theorem using $\CSL$ as its combinatorial core, but with no optimal big Ramsey degree bounds. Anglès d'Auriac, Liu, Mignoty and Patey~\cite{angles2023carlson} noticed that his proof could be formalized in~$\RCA_0$.

\begin{definition}
Let $\GG = (\{0\}^{<\omega, 1}, E)$ be the graph where $E$ is the symmetric binary relation defined as follows:
for $v, w \in \{0\}^{<\omega, 1}$, $v E w$ if $|v| \neq |w|$ and, assuming $|v| < |w|$, the following holds:
\begin{itemize}
    \item[(1)] Passing number property: $w(|v|) = x_0$;
    \item[(2)] Triangle-freeness condition: there is no $i < |v|$ with $v(i) = w(i) = x_0$.
\end{itemize}
\end{definition}

Hubi\v{c}ka~\cite{hubivcka2020big} proved that for every unordered $\omega$-variable word $W$ over~$\{0\}$,
$(\USub^{1,\star}(W), E)$ is computably isomorphic to the triangle-free Henson graph~$\HH_3$.
Moreover, these isomorphisms hold provably over $\RCA_0$, so the indivisibility of the triangle-free Henson graph follows from Carlson-Simpson lemma for 1-variable words over $\RCA_0$. In what follows, we identify the graph $\HH_3$ with its set of vertices:

\begin{theorem}[$\RCA_0 + \CSL^1(1)$]\label[theorem]{thm:henson1-csl1}
For any integer $k > 0$ and any finite coloring $\chi : \HH_3 \to k$, there exists $f \in \binom{\HH_3}{\HH_3}$ such that $\chi$ uses exactly one color on $f(\HH_3)$.
\end{theorem}

By \Cref{thm:csl12-pi04-conservation} and \Cref{cor:csl12-separation-aca}, it follows that the previous theorem is $\forall \Pi^0_4$-conservative over~$\RCA_0 + \BSig_2$ and does not imply~$\ACA_0$. The latter fact was recently proven by Liu and Patey~\cite{liu2026reverse}.

\subsection{Tree theorem for pairs}

Chubb, Hirst, and McNicholl~\cite{chubb2009reverse} introduced a partition theorem for comparable nodes over perfect binary trees, which was extensively studied in reverse mathematics. Given a set~$S \subseteq 2^{<\NN}$ and $n \in \NN$, we write $[S]^n$ for the set of all pairwise comparable $n$-subsets of~$S$, that is, the subsets of the form $\{\sigma_0, \dots, \sigma_{n-1} \}$ such that $\sigma_i \preceq \sigma_{i+1}$.

\begin{definition}
A set~$T \subseteq \bstr$ is \emph{order-isomorphic to $\bstr$} (written $T \cong \bstr$) if there is a bijection $g : T \to \bstr$ such that for every~$\sigma, \tau \in T$, $\sigma \preceq \tau$ iff $g(\sigma) \preceq g(\tau)$.
\end{definition}

\begin{statement}
$\TT^n_k$ is the statement \qt{For every coloring $f : [\bstr]^n \to k$, there is a set~$T \cong \bstr$ such that $[T]^n$ is $f$-monochromatic}. 
\end{statement}

Chubb, Hirst, and McNicholl~\cite{chubb2009reverse} proved that $\TT^n_2$ is between $\ACA_0$ and $\RT^n_2$ over~$\RCA_0$. In particular, for $n \geq 3$, $\TT^n_2$ is equivalent to~$\ACA_0$. Patey~\cite{patey2016strength} and Dzhafarov and Patey~\cite{dzhafarov2017coloring} proved that $\TT^2_2$ is strictly in-between $\ACA_0$ and $\RT^2_2$ over $\omega$-models. Chong, Li, Liu and Yang~\cite{chongStrengthRamseyTheorem2019a,chongStrengthRamseyTheorem2019b,chongStrengthRamseyTheorem2019c} published a series of article and proved, among other things, that $\TT^2_2$ does not imply $\WKL_0$ over~$\RCA_0$. 

From a proof-theoretic viewpoint, Corduan, Groszek and Mileti~\cite{corduan2010reverse} showed that $\BSig_2$ is strictly weaker than $\TT^1$ over~$\RCA_0$, while Chong, Li, Wang and Yang~\cite{chong2020strength} proved that $\TT^1$ is $\Pi^1_1$-conservative over $\BSig_2 + \mathsf{P}\Sigma^0_1$, where $\mathsf{P}\Sigma^0_1$ is equivalent to the well-foundedness of $\bbomega^{\bbomega}$. In particular, $\TT^1$ does not imply $\ISig_2$ over~$\RCA_0$.
Last, Chong, Wang and Yang~\cite{chong2023conservation} proved that $\TT^1$ is $\forall \Pi^0_3$-conservative over $\RCA_0$. Note that $\TT^2_2$ implies $\TT^1$ over~$\RCA_0$.

\begin{proposition}
$\RCA_0 \vdash \forall k (\CSL^1_k \to \TT^2_k)$.
\end{proposition}
\begin{proof}
Let $f : \bstr \to k$ be an instance of $\TT^2_k$ for some~$k \geq 1$.
Let $A = \{0,1\}$ and $h : A^{<\omega, 1} \to k$ be defined by
$$h(v) = f(v[\epsilon], v[0])$$
By $\CSL^1_k$, there is an infinite $\omega$-variable word~$W$ over~$A$ such that
$\USub^{1,\star}_A(W)$ is $h$-monochromatic for some color~$c$. Let 
$$T = \{ \sigma \in \bstr : |\sigma| \text{ is even } \wedge \forall i < \frac{|\sigma|}{2}, \sigma(2i) = 0 \}$$
Note that $T \cong \bstr$.
Finally, let $S = \{ W[\sigma] : \sigma \in T \}$.
Since $\sigma \mapsto W[\sigma]$ is order-preserving, $S \cong \bstr$.
We claim that $[S]^2$ is $f$-monochromatic for color~$c$.
Fix some $\hat \sigma, \hat \tau \in S$ such that $\hat \sigma \prec \hat \tau$.
Let $\sigma, \tau \in T$ be such that $\hat \sigma = W[\sigma]$ and $\hat \tau = W[\tau]$. In particular, $\sigma \prec \tau$. By definition of~$T$, there is some~$\tau' \in \bstr$ such that $\tau = \sigma \cdot 0 \cdot \tau'$.
Let $u = \sigma \cdot \star \cdot \tau'$. By choice of~$W$, $h(W[u]) = c$.
By definition of~$h$, $$f(\hat \sigma, \hat \tau) = f(W[\sigma], W[\tau]) = f(W[u][\epsilon], W[u][0]) = h(W[u]) = c$$
\end{proof}

Chong, Li, Wang and Yang~\cite{chong2020strength} asked whether $\TT^2_2$ implies $\ISig_2$ over~$\RCA_0$. We answer negatively with the following corollary:

\begin{corollary}
$\WKL_0 + \TT^2_2$ is $\forall \Pi^0_4$-conservative over~$\RCA_0 + \BSig_2$.
\end{corollary}

\section{Conclusion and open questions}\label[section]{sect:open-questions}

The study of variable word theorems from the viewpoint of reverse mathematics remains a subject with a relative lack of understanding, as witnessed by the big gap between lower bounds and upper bounds. For instance, Hindman's theorem,  $\OVW^1$ and $\CSL^2$ are known to imply $\ACA_0$ over~$\RCA_0$, but the best known upper bounds are $\ACA_0^+$ for $\HT$, and $\PIOOCA_0$ for~$\OVW^1$ (Liu, unpublished). The following questions are of fundamental importance:

\begin{question}
Does $\ACA_0 \vdash \HT$?
\end{question}

\begin{question}
For every~$n \in \omega$, does $\ACA_0 \vdash \OVW^n$?
\end{question}

The framework of $\alpha$-largeness has been shown to be successful in proving separations from theories, where techniques over $\omega$-models failed~\cite{houerou2024pi}. The following questions might be more tractable and would still yield significant insights on the reverse-mathematical strength of variable word theorems:

\begin{question}
Is $\RCA_0 + \HT$ $\forall \Pi^0_3$-conservative over~$\ACA_0$?
\end{question}

\begin{question}
For every~$n \in \omega$, is $\RCA_0 + \OVW^n$ $\forall \Pi^0_4$-conservative over~$\ACA_0$?
\end{question}

Last, the known lower bounds to $\CSL^n$ are all coming from the fact that $\RCA_0 \vdash \forall \ell(\CSL^n_\ell \to \RT^{n+1}_\ell)$ (see \Cref{prop:rt-to-csl}). It follows that its level version, $\RFCSL^n$, has no known non-trivial lower bound.

\begin{question}
For every~$n \in \omega$, is $\RCA_0 + \RFCSL^n$ $\Pi^1_1$-conservative over~$\RCA_0$?
\end{question}

\bibliographystyle{plain}
\bibliography{biblio}

\end{document}